\documentclass{article}
\usepackage{subfiles}
\usepackage[utf8]{inputenc}
\usepackage{amsmath}
\usepackage{amsthm}
\usepackage{stmaryrd}
\usepackage{amsfonts}
\usepackage{amssymb,color}
\usepackage{graphicx}
\usepackage[left=2.2cm,right=2.2cm,top=2.6cm,bottom=2.6cm]{geometry}
\usepackage{titlesec}
\usepackage{enumitem}
\usepackage{calc}
\usepackage{caption}
\usepackage{mathrsfs}
\usepackage{dsfont}

\usepackage{multicol}
\usepackage{mathtools}
\RequirePackage[colorlinks,citecolor=blue,urlcolor=blue]{hyperref}

\usepackage{tikz}
\usetikzlibrary{arrows,calc,quotes,angles}
\numberwithin{equation}{section}
\makeatletter
\newtheorem*{rep@theorem}{\rep@title}
\newcommand{\newreptheorem}[2]{%
\newenvironment{rep#1}[1]{%
 \def\rep@title{#2 \ref{##1}}%
 \begin{rep@theorem}}%
 {\end{rep@theorem}}}
\makeatother
\numberwithin{equation}{section}
\numberwithin{figure}{section}
\newtheorem{theorem}{Theorem}
\newtheorem{definition}[theorem]{Definition}
\newtheorem{proposition}[theorem]{Proposition}
\newtheorem{lemma}[theorem]{Lemma}

\newtheorem{corollary}[theorem]{Corollary}

\titleformat{\paragraph}
{\normalfont\normalsize\bfseries}{\theparagraph}{1em}{}
\titlespacing*{\paragraph}
{0pt}{3.25ex plus 1ex minus .2ex}{1.5ex plus .2ex}

\allowdisplaybreaks

\title{Simulation of Reflected Brownian motion on two dimensional wedges}
\author{Pierre Bras\footnote{Sorbonne Universit\'e, Laboratoire de Probabilit\'es, Statistique et Mod\'elisation, UMR 8001, case 188, 4 pl. Jussieu, F-75252 Paris Cedex 5, France. E-mail: \texttt{pierre.bras@sorbonne-universite.fr}. Supported in part by \'Ecole Normale Sup\'erieure, D\'epartement de Math\'ematiques et Applications.}$\ $ and Arturo Kohatsu-Higa\footnote{Ritsumeikan University. Department of Mathematical Sciences. E-mail: \texttt{khts00@fc.ritsumei.ac.jp}. Supported in part by KAKENHI 20K03666.}}
\date{}

\begin{document}

\maketitle

\begin{abstract}
	We study a correlated Brownian motion in two dimensions, which is reflected, stopped or killed in a wedge represented as the intersection of two half spaces. First, we provide explicit density formulas, hinted by the method of images. These explicit expressions rely on infinite oscillating sums of Bessel functions and may demand computationally costly procedures. 
	
	 We propose suitable recursive algorithms for the simulation of the laws of reflected and stopped Brownian motion which are based on generalizations of the reflection principle in two dimensions.
	We study and give bounds for the complexity of the proposed algorithms. 
	
	\smallskip
	
	
\end{abstract}

\section{Introduction}

\indent The objective of the present article is to provide exact formulas and algorithms for the simulation of a normally reflected two-dimensional Brownian motion starting at $ x_0\in\mathbb{R}^2$.

The reflected process is denoted by $X\equiv (X_t)_{t\in [0,T]}$ and the domain of reflection is a wedge $\mathcal{D}$, i.e. a subset of $\mathbb{R}^2$ delimited by two (non parallel) lines, so as to give Monte Carlo methods for the estimation of $\mathbb{E}^{x_0}[f(X_T)]$, where $f : \mathcal{D} \rightarrow \mathbb{R}$ is a measurable function such that $f(X_T) \in L^1$ and $T>0$ is the finite time horizon.
Our first goal is to obtain an explicit formula for the density of  $ X_t $  (see Theorem \ref{Rformula_alpha_pi_m}). Unfortunately, this formula involves oscillating infinite sums of Bessel functions. 
Instead of directly computing these sums, we propose an alternative simulation method which uses an extension of the reflection principle in two dimensions for a particular type of wedges with angle $\pi/m$, $m \in \mathbb{N}$.
As a first step, we obtain a simulation method for $ \mathbb{E}^{x_0}[ f(W_{T \wedge \tau})]$, where $W=(W_t)_{t\geq 0}$ is a two-dimensional Brownian motion and $ \tau $ stands for the first time the process $ W $ touches the boundary of the wedge $\mathcal{D}$.  
Applying the methodology for the stopped process recursively one obtains  an algorithm for the reflected process $ X $.
The algorithms we present here rely on an idea which has links to the rectangle method \cite{lejay2006}: we include a smaller wedge of angle $\pi/m$ inside the first wedge and simulate the exit from the smaller wedge.

 The algorithm for the simulation of the stopped process improves the method proposed in \cite{metzler2008}. In the reflected case, we prove that although some lower moments are finite, the average number of simulations for this method is infinite. This effect is created due to the large number of reflections that may happen if the process enters the corner of the wedge. We then propose a modified approximation scheme which takes into account the asymptotic behaviour of the density near the corner.
With this modification, we show that the expected number of iterations of the algorithm is finite and measure the error of approximation in total variation distance. We then give adaptations of our algorithms to more general It\={o} processes reflected in a wedge.

The expression of the density of the stopped Brownian motion in a wedge of $\mathbb{R}^2$ has been obtained in \cite{iyengar1985} using the method of images. However, the author does not provide full arguments for the verification of the initial condition and  some exchanges of infinite summation and integrations are not fully explained. For the history of the expression of the density of the stopped Brownian motion on a wedge, see \cite{Ban}.

In \cite{chupeau2015} and \cite{metzler2010} the authors correct some mistakes in formulas appearing in \cite{iyengar1985} and give formulas for a number of other random variables related to the stopped Brownian motion in two dimensions with general correlation coefficients, such as the survival probability and the first-passage time distribution.
However, formulas from \cite{iyengar1985} and \cite{metzler2010} are not directly applicable to simulation algorithms of stopped processes, as they involve infinite sums and Bessel functions. \cite[Section 2]{metzler2008} gives an approximate simulation algorithm using the semi-analytic expression of the density established in \cite{iyengar1985}. But this method implies the approximation of an infinite sum of Bessel functions and it leads to bias when simulating the stopping time $\tau$. 


More recent papers (\cite{escobar2013}, \cite{blanchetscalliet2013}) extend the results of \cite{iyengar1985} and \cite{metzler2010} to three dimensions and give algorithms and density formulas for the stopped Brownian motion in three dimensions, using the method of images in $\mathbb{R}^3$. These algorithms which point towards the possibility of larger dimension extensions are however limited to some specific values of correlation. 
In \cite{kaushansky2017}, the authors use the spectral decomposition method which applies in a larger generality but still relies on efficient computational methods for eigenvalue calculations and truncations of infinite sums. 
These references concentrate on the stopped Brownian motion case and there is no discussion about the simulation of reflected Brownian motion in wedges.

Simulation algorithms for the reflected Brownian motion have been only partly studied, although they share similarities with the stopped case.  Simulation algorithms for quantities related to reflected Brownian motion on an orthant in multi-dimensions can be found in \cite{blanchet2014} and \cite{blanchet2015}. They do not use the method of images, but the characterization through the Skorokhod problem and the so-called $ \varepsilon $-strong methodology which relies on the regularity of the reflection functional. Still, the method proposed has a theoretically infinite mean running time.

Other approximations methods which are closer to random walk versions of the reflected Brownian motion on tile domains are studied in \cite{Dubedat} and \cite{Kager}. Furthermore, other simulations methods for reflected Brownian motion which do not cover the case of the wedge can be found in \cite{Cos}, \cite{Bossy}, \cite{Burzdy} and \cite{Bayer} the references therein.

The reflected Brownian motion and more generally reflected processes have applications in finance (for example, in stochastic models where the process, which can model an interest rates, is constrained to be non-negative \cite{ha2009} or has other constraints like barriers such as in \cite{Ichiba}); in queueing models \cite{giorno1986}, etc. In particular, in the recent years there has been a lot of developments concerning the study of stationary measures of reflected Brownian motion such as \cite{DM} which is slightly related to the ideas which we use here to construct a simulation method on a wedge. For a review, on this topic and the relationship with queuing theory see e.g. \cite{dieker}.

Considering the particular case of a wedge may open the way to new simulation algorithms for reflected processes adapted to non-smooth domains with corner points in two or more dimensions (see Section \ref{sec:higher_dim}).

\medskip

The paper is organized as follows.
In Section \ref{section:setting} we state the framework of the problem, giving a parametrization of the wedge in $\mathbb{R}^2$.
We also give the formula for the density of the stopped Brownian motion in a wedge of angle $\alpha \in (0,2\pi)$. The density of the stopped Brownian motion had already been obtained in \cite{iyengar1985}. A more general case is studied in \cite{Ban} which provides a detailed history and a full proof that the density formula is the fundamental solution of the associated partial differential equation.

In Section \ref{section:density_formulas}, we provide the density of the reflected Brownian motion on the wedge. First, we use the method of images in a special type of wedges as it serves to understand how one induces the density of the reflected Brownian motion in the general case.  In the Appendix we include the proof in the reflected case on more general cones following \cite{Ban}.


In Section \ref{section:algorithms}, we give explicit algorithms for the simulation of the stopped and reflected Brownian motions. To do so, we first derive from \cite{metzler2010} and the formulas proved in Section \ref{section:density_formulas} expressions of densities of the exit time and  exit point from the wedge. Then, we give an algorithm for the simulation of the stopped Brownian motion using the explicit expressions in the case $\alpha = \pi/m$. The simulation algorithm is based on the simulation of a random sequence of stopped Brownian motion processes on domains with angle $\alpha = \pi/m$.
This first algorithm is useful for the simulation of reflected Brownian motion, but it also has its own interest. Then we give the simulation algorithm for the reflected Brownian motion with unit diffusion matrix.

In Section \ref{section:complexity}, we study the mean number of iterations of the algorithm to finish, denoted $\mathbb{E}[N]$. For the stopped Brownian motion, $\mathbb{E}[N]$ is bounded above by a constant. The reflected case is more difficult to deal with. In this case, $N$ is equal to the number of times a Brownian motion goes from one frontier of a wedge to another. We first give theoretical bounds for $\mathbb{E}[N]$ and prove that although the number of iterations is almost surely finite, $\mathbb{E}[N]$ is infinite, which is due to the events when the Brownian motion comes close to the origin. We then propose a modification of the algorithm: we stop the algorithm when the Brownian motion comes too close to the origin, measured by a parameter $ \varepsilon $ and use an approximation of the density for the reflected Brownian motion close to the origin, induced by the density formula in Section \ref{section:density_formulas}. We then prove that for the modified algorithm, $\mathbb{E}[N]$ has an upper bound which grows as $\varepsilon^{1-p}$ for all $ p\in (1,2) $, and that the error of approximation is of order $\varepsilon$.

In Section \ref{section:extensions}, we show how the algorithms can directly be applied for the simulation of the reflected and stopped Brownian motions with non-zero drift. Then we give an algorithm for the simulation of a class of It\={o} diffusion processes, reflected or stopped in a wedge in $\mathbb{R}^2$. This algorithm is in fact a direct application of the precedent algorithms, combined with a Euler-Maruyama scheme.

Finally, in Section \ref{section:simulations}, we perform Monte Carlo simulations for the estimation of reflected and stopped Brownian motions and It\={o} processes. We show that the bias in the algorithm from \cite{metzler2008} can be significant compared with our method which is exact. Simulations prove the need for the approximation algorithm for the reflected Brownian motion, as the exact algorithm takes too much time to be efficient. Clearly, one remaining subject which is not considered here is the reflecting case with non-unit diffusion coefficient. This remains an open problem. 

The authors would like to express thanks to the anonymous referee which provided references \cite{Ban} and \cite{DM} which improved the results and shortened the proof of Theorem \ref{Rformula_alpha_pi_m}.

\section{Notations}

We give a brief list of notations that are used through the text.

The natural logarithm is denoted with $ \log $.
We frequently use the notation $ \pm $ in order to denote the two rays that define a wedge. Sometimes rather than writing two equations we just write one using the symbol $ \pm $ or $ \mp $ meaning that we are stating two equations, one using the top symbols and another using the bottom symbols that appear throughout the equation. 

$ I_a $ stands for the modified Bessel function of the first kind given for $a \ge 0, \ x \ge 0$ by:
\begin{equation}
\label{eq:10}
I_a(x) = \sum_{k=0}^\infty \frac{x^{a+2k}}{2^{a+2k}k!\Gamma({k+a+1})}.
\end{equation}
This function satisfies the differential equation
\begin{equation}
\label{bessel_equation}
x^2 \frac{d^2 I_a(x)}{dx^2} + x \frac{d I_a(x)}{dx} - (x^2 + a^2) I_a(x) = 0, \ x \ge 0.
\end{equation}

We consider the space $\mathbb{R}^2$ endowed with the Euclidean norm denoted by $| \boldsymbol{\cdot} |$.

For $\mathcal{D} \subset \mathbb{R}^2$ and $k \in \mathbb{N}$, we denote by $\mathcal{C}_b^k(\mathcal{D})$ the set of real-valued functions defined on $\mathcal{D}$ which are bounded and have bounded partial derivatives up to the order $k$.

Now we define the probabilistic setting as follows. Let $(\Omega,\mathcal{F},\mathbb{P},(\mathcal{F}_t)_{t \ge 0})$ be a filtered probability space satisfying the usual conditions. Let $T >0$ be the time horizon and let $W=(W_t)_{t \ge 0}$ be a two-dimensional correlated Brownian motion on $(\Omega,\mathcal{F},\mathbb{P},(\mathcal{F}_t)_{t \ge 0})$. We denote by $\Sigma$ the covariance matrix of $W$
and we assume that $\Sigma$ is a non-singular matrix.
We define a family of probabilities on $(\Omega,\mathcal{F},(\mathcal{F}_t)_{t \ge 0})$ by:
$$ \forall x_0 \in \mathbb{R}^2, \ \mathbb{P}^{x_0} := \mathbb{P}_{|W_0 = x_0} .$$
We will denote by $\mathbb{E}^{x_0}$ the expectation under $\mathbb{P}^{x_0}$.

If $W=(W_t)_{t\geq 0}$ is a Brownian motion we will say ``stopped Brownian motion" to designate the process $(W_{t \wedge \tau})_{t\geq 0}$, and ``killed Brownian motion" to designate the process $(W_t \mathds{1}_{\tau > T})_{t\in [0,T]}$, where $\tau$ is some stopping time defined later.

We will frequently use in the proofs the Bessel process
$(R_t)_{t\geq 0}$ in dimension $\delta \ge 2$, i.e. a real-valued process satisfying the stochastic differential equation:
$$ R_t = r_0+B_t + \int_0^t\frac{\delta-1}{2}\frac{ds}{R_s} ,$$
where $(B_t)_{t\geq 0}$ is a standard one-dimensional Brownian motion starting from zero. The index of the Bessel process is defined as $ \nu=\frac{\delta}{2}-1$. 
 For a general reference to these matters we refer to \cite[Section 6]{jeanblanc2010}.

\section{Setting of the problem}
\label{section:setting}
\label{sec:3.1}

Let $a \in\mathbb{R}\cup\{\pm\}$ and consider the subset $\mathcal{D} \subset \mathbb{R}^2$ defined by one of the four following cartesian equations:
\begin{align}
\label{def:wedge}
\mathcal{D} =
\left\lbrace \begin{array}{ll}
\lbrace (x,y) \in \mathbb{R}^2, \ y \ge 0 \text{ and } y \le ax \rbrace & \text{ with } a > 0, \\
\lbrace (x,y) \in \mathbb{R}^2, \ y \ge 0 \text{ and } y \ge ax \rbrace & \text{ with } a < 0, \\
\lbrace (x,y) \in \mathbb{R}^2, \ y \ge 0 \text{ or } y \ge ax \rbrace & \text{ with } a > 0 , \\
\lbrace (x,y) \in \mathbb{R}^2, \ y \ge 0 \text{ or } y \le ax \rbrace & \text{ with } a < 0.
\end{array} \right.
\end{align}
The set $\mathcal{D}$ is called a wedge\footnote{The cases where $ a=0 $ are not considered here because they lead to cases where the problem can be simplified to a one dimensional problem.}. The angle of the wedge is denoted by $\alpha$ and chosen such that $\alpha \in (0,2\pi)$. An example is given in Figure \ref{wedge_figure}. 
We write the wedge $\mathcal{D}$ in polar coordinates as follows:
$$ \mathcal{D} = \lbrace (r\cos(\theta),r\sin(\theta)), \ r \in \mathbb{R}^+, \ \theta \in [0,\alpha] \rbrace .$$

We define its boundary as $ \partial \mathcal{D} = \partial \mathcal{D}^- \cup \partial \mathcal{D}^+$, where
\begin{align*}
\partial \mathcal{D}^- := & \lbrace (x,y) \in \mathcal{D}, y = 0 \rbrace = \lbrace (r\cos(\theta),r\sin(\theta)), \ r \in \mathbb{R}^+, \ \theta =0 \rbrace, \\
\partial \mathcal{D}^+ := & \lbrace (x,y) \in \mathcal{D}, y = ax \rbrace =\lbrace (r\cos(\theta),r\sin(\theta)), \ r \in \mathbb{R}^+, \ \theta =\alpha \rbrace.
\end{align*}

In general, we use the notation $\mathcal{D}=\langle \alpha\rangle  $ to denote a wedge which boundary is determined by the rays $ \partial\mathcal{D}^\pm$. In a similar way, we also generalize this notation to a general wedge which boundaries are determined by the rays at the angles $ 0\leq \alpha<\beta<2\pi $ as $ \mathcal{V}=\langle\alpha,\beta\rangle $.
\begin{figure}
	\centering
\begin{tikzpicture}
\coordinate (origin) at (0,0);
\coordinate (e1) at (1,0);
\coordinate (a) at (-1.3,2);
\draw[thick,->] (0,-0.7) -- (0,3);
\draw[thick,->] (-1.4,0) -- (3,0);
\draw[very thick,red] (0,0) -- (2.3,0);
\draw[very thick, red] (0,0) -- (a);
\draw[thick,blue,->] (1.6,0) -- (1.6,1) node[anchor=south] {$n(x_1)$};
\filldraw[black] (1.6,0) circle (2pt) node[anchor=north] {$ x_1 $};
\draw[thick,blue,->] ($(0,0)!0.8!(a)$) -- ($(0,0)!0.8!(a) + (0,0)!0.4!(2,1.3)$) node[anchor=south] {$n(x_2)$};
\filldraw[black] ($(0,0)!0.8!(a)$) circle (2pt) node[anchor=north] {$ x_2 $};
\draw pic[draw,angle radius=1cm,"$\alpha$" anchor = west] {angle = e1--origin--a};
\end{tikzpicture}
	\caption{Example of a wedge of angle $\alpha$.}
	\label{wedge_figure}
\end{figure}
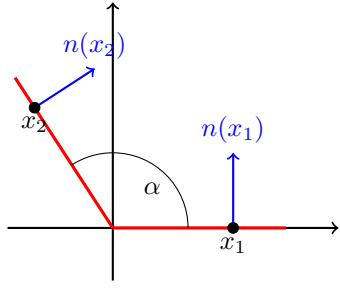

We use the following stopping times:
\begin{align*}
\tau & = \inf \ \lbrace t>0, W_t \in \partial \mathcal{D} \rbrace. \\
\tau^\pm & = \inf \ \lbrace t>0, W_t \in \partial \mathcal{D}^\pm \rbrace.
\end{align*}

As in \cite{metzler2010}, we parametrize the covariance matrix of $W$ via its Cholesky decomposition and write $\Sigma = \sigma \sigma^T$ with
$$ \sigma = \begin{pmatrix}
\sigma_1 \sqrt{1-\rho^2} & \sigma_1 \rho \\
0 & \sigma_2
\end{pmatrix}, $$
where $\rho \in (-1,1)$ and $\sigma_1, \sigma_2 \ge 0$. Then assuming that $ \sigma $ is invertible, we consider the process $W' := \sigma^{-1}W$, which is a standard two-dimensional Brownian motion with independent components. Furthermore, using the explicit formula for the inverse matrix, we deduce that
\begin{align*}
W^2 = 0 & \iff W'^2 = 0, \\
W^2 = aW^1 & \iff \left\lbrace
	\begin{array}{ll}
		W'^2 = \frac{a\sigma_1\sqrt{1-\rho^2}}{\sigma_2-a\sigma_1\rho}W'^1 & \text{ if } \sigma_2-a\sigma_1\rho \ne 0, \\
		W'^1 = 0 & \text{ if } \sigma_2-a\sigma_1\rho = 0.
	\end{array} \right.
\end{align*}
In the case that $\sigma_2-a\sigma_1\rho = 0$, then the problem reduces to two independent one-dimensional Brownian motions being reflected or stopped in the first quadrant, and therefore the problem can be reduced to using two independent reflected or stopped Brownian motions in one dimension. For explicit formulas, see \eqref{one_dimension_reflected} and \eqref{one_dimension_killed}. Thus in the following we assume that $\sigma_2-a\sigma_1\rho \ne 0$.
So $\tau^-$ is the first passage time of $W'$ through the horizontal axis and $\tau^+$ is the time of the first passage of $W'$ through the line $y = a' x$, where
\begin{equation}
\label{eq:a_prime}
a' = \frac{a\sigma_1\sqrt{1-\rho^2}}{\sigma_2-a\sigma_1\rho} ,
\end{equation}
so that $\tau=\tau^-\wedge\tau^+$ is the exit time of $W'$ from the wedge $\mathcal{D}'=\langle \alpha'\rangle $, where $ \alpha' $ is defined as $\arctan(a')$ or $\pi + \arctan(a')$ or $2\pi + \arctan(a)$ depending on the cases. The process $W'$ starts at $x'_0 = \sigma^{-1}\cdot x_0$. The possible cases are summed up in Table \ref{figure:wedge_four_cases}.

After the above transformation the original wedge is transformed into a new one according to this table.
Therefore mutatis-mutandis, we shall omit the primes in the notation and without loss of generality we directly assume that the process $W$ has the correlation matrix $\Sigma = I_2$ and the wedge is given as in \eqref{def:wedge} with angle $ \alpha $.

\begin{figure}
\centering
\begin{tabular}{|c|c|c|c|c|}
\hline
\rule[-1ex]{0pt}{2.5ex} \shortstack{Cartesian \\ Equation} & \shortstack{$\lbrace y \ge 0 \text{ and } y \le ax \rbrace,$ \\ $ a > 0$} & \shortstack{$\lbrace y \ge 0 \text{ and } y \ge ax \rbrace,$ \\ $a < 0 $} & \shortstack{$\lbrace y \ge 0 \text{ or } y \ge ax \rbrace,$ \\ $a > 0$} & \shortstack{$\lbrace y \ge 0 \text{ or } y \le ax \rbrace,$ \\ $a < 0$} \\
\hline
\rule[-1ex]{0pt}{2.5ex} Value of $\alpha$ & $\arctan(a)$ & $\pi + \arctan(a)$ & $\pi + \arctan(a)$ & $2\pi + \arctan(a)$ \\
\hline
\rule[-1ex]{0pt}{2.5ex} \shortstack{Wedge for \\the original \\ problem  \\ \textcolor{white}{a} \\ \textcolor{white}{a} \\ \textcolor{white}{a} \\ \textcolor{white}{a} } &
\begin{tikzpicture}[scale=0.6, transform shape]
\coordinate (origin) at (0,0);
\coordinate (e1) at (1,0);
\coordinate (a) at (1.3,2);
\draw[thick,->] (0,-2) -- (0,3);
\draw[thick,->] (-2,0) -- (3,0);
\draw[very thick,red] (0,0) -- (2.3,0);
\draw[very thick, red] (0,0) -- (a);
\draw pic[draw,angle radius=1cm,"$\alpha$" anchor = west] {angle = e1--origin--a};
\end{tikzpicture} &
\begin{tikzpicture}[scale=0.6, transform shape]
\coordinate (origin) at (0,0);
\coordinate (e1) at (1,0);
\coordinate (a) at (-1.3,2);
\draw[thick,->] (0,-2) -- (0,3);
\draw[thick,->] (-2,0) -- (3,0);
\draw[very thick,red] (0,0) -- (2.3,0);
\draw[very thick, red] (0,0) -- (a);
\draw pic[draw,angle radius=1cm,"$\alpha$" anchor = west] {angle = e1--origin--a};
\end{tikzpicture} &
\begin{tikzpicture}[scale=0.6, transform shape]
\coordinate (origin) at (0,0);
\coordinate (e1) at (1,0);
\coordinate (a) at (-1.3,-2);
\draw[thick,->] (0,-2) -- (0,3);
\draw[thick,->] (-2,0) -- (3,0);
\draw[very thick,red] (0,0) -- (2.3,0);
\draw[very thick, red] (0,0) -- (a);
\draw pic[draw,angle radius=1cm,"$\alpha$" anchor = west] {angle = e1--origin--a};
\end{tikzpicture} &
\begin{tikzpicture}[scale=0.6, transform shape]
\coordinate (origin) at (0,0);
\coordinate (e1) at (1,0);
\coordinate (a) at (1.3,-2);
\draw[thick,->] (0,-2) -- (0,3);
\draw[thick,->] (-2,0) -- (3,0);
\draw[very thick,red] (0,0) -- (2.3,0);
\draw[very thick, red] (0,0) -- (a);
\draw pic[draw,angle radius=1cm,"$\alpha$" anchor = west] {angle = e1--origin--a};
\end{tikzpicture} \\
\hline
\rule[-1ex]{0pt}{2.5ex} \shortstack{Equation \\ of $W'$ if \\ $\sigma_2 - a \sigma_1 \rho > 0$} & \shortstack{$\lbrace y \ge 0 \text{ and } y \le a'x \rbrace,$ \\ $a'>0$}  & \shortstack{$\lbrace y \ge 0 \text{ and } y \ge a'x \rbrace,$ \\ $a'<0$} & \shortstack{$\lbrace y \ge 0 \text{ or } y \ge a'x \rbrace,$ \\ $a'>0$} & \shortstack{$\lbrace y \ge 0 \text{ or } y \le a'x \rbrace,$ \\ $a'<0$} \\
\hline
\rule[-1ex]{0pt}{2.5ex} \shortstack{Equation \\ of $W'$ if \\ $\sigma_2 - a \sigma_1 \rho <0$} & \shortstack{$\lbrace y \ge 0 \text{ and } y \ge a'x \rbrace,$ \\ $a'<0$} & \shortstack{$\lbrace y \ge 0 \text{ and } y \le a'x \rbrace,$ \\ $a'>0$} & \shortstack{$\lbrace y \ge 0 \text{ or } y \le a'x \rbrace,$ \\ $a'<0$} & \shortstack{$\lbrace y \ge 0 \text{ or } y \ge a'x \rbrace,$ \\ $a'>0$} \\
\hline
\end{tabular}
\caption{Table summing up the four different cases of wedge $\mathcal{D}$}
\label{figure:wedge_four_cases}
\end{figure}

\bigskip

In some results, we assume that the angle of the wedge satisfies the condition:
$$\alpha = \frac{\pi}{m} ,$$
for some $m \in \mathbb{N}$.
We define adjacent wedges as follows. For $k = 0, 1, ..., 2m-1$, let
$$ \mathcal{D}_k := \lbrace (r\cos(\theta),r\sin(\theta)): \ r \in [0,+\infty), \ \theta \in [k\alpha, (k+1)\alpha] \rbrace.$$
We remark that $\mathcal{D}_0 = \mathcal{D}$ and that $ \mathbb{R}^2 = \bigcup_{k=0}^{2m-1} \mathcal{D}_k $. 
We also define the transformation $T_k$:
\begin{align*}
	T_k : & \mathcal{D}_0 \longrightarrow \mathcal{D}_k \\
	T_k((r \cos \theta, r \sin \theta)):=&(r \cos(\vartheta_k), r \sin(\vartheta_k))\\
	\vartheta_k:=&
	\left\lbrace
	\begin{array}{cc}
		(k+1)\alpha-\theta;& k\text{ odd}, \\ k\alpha+\theta;& k\text{ even}.
	\end{array}\right.
\end{align*}

From the definition, it follows that $T_k$ is an isometric bijection between $\mathcal{D}_0$ and $\mathcal{D}_k$. This setting is represented in Figure \ref{fig:partition_Dk}.

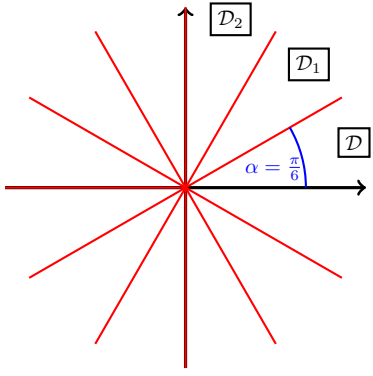
\begin{figure}
	\label{figure_method_images}
	\centering
	\begin{tikzpicture}[thick,scale=0.6, every node/.style={scale=0.8}]
		\coordinate (origin) at (0,0);
		\coordinate(e1) at (4,0);
		\draw[very thick,->] (0,-4) -- (0,4);
		\draw[very thick,->] (-4,0) -- (4,0);
		
		\foreach \k in {1,...,11}
		{
			\draw[red] (origin) -- ([rotate around={\k*30:(origin)}]e1);
		}
		
		\coordinate (f1) at ([rotate around={30:(origin)}]e1);
		\draw[blue] pic[draw,angle radius=2cm,"$\hspace{-0.3cm}\alpha = \frac{\pi}{6}$" anchor = west] {angle = e1--origin--f1};
		\node[draw] at ($(e1)!0.5!(f1)$) {$\mathcal{D}$};
		
		\coordinate (f2) at ([rotate around={60:(origin)}]e1);
		\node[draw] at ($(f1)!0.5!(f2)$) {$\mathcal{D}_1$};
		
		\coordinate (f3) at ([rotate around={90:(origin)}]e1);
		\node[draw] at ($(f2)!0.5!(f3)$) {$\mathcal{D}_2$};

	\end{tikzpicture}
	\caption{Partition of $\mathbb{R}^2$ using $\{\mathcal{D}_k\}_k$}
	\label{fig:partition_Dk}
\end{figure}

\bigskip

The density of the killed Brownian motion and some related random quantities appear in \cite{metzler2010}. We first quote the explicit density result for killed Brownian motion which is fully proven in \cite[Lemma 1]{Ban} for generalized multidimensional cones.
\begin{theorem}[Killed case]
	\label{Kformula_alpha_pi_m}Assume that $\mathcal{D}=\langle \pi/m \rangle $ for some $m \in \mathbb{N}$. Then we have the following formula for the density of  killed Brownian motion on the wedge $ \mathcal{D} $:
	\begin{align}
		\label{alpha_pi_m_killed}
		\forall t>0, \ x, y \in \mathcal{D}, \ \mathbb{P}^{x}(W_t \in dy, \  \tau > t) & = \frac{1}{2\pi t} \sum_{k=0}^{2m-1} (-1)^k e^{-\frac{|x-T_ky|^2}{2t}} dy .
	\end{align}
	In general, let $\alpha \in (0,2\pi)$, $ \mathcal{D}=\langle \alpha\rangle$ and  $x = (r_0 \cos(\theta_0),r_0 \sin(\theta_0))\in \mathcal{D}$. Then for any  $y=(r\cos(\theta),r\sin(\theta))\in \mathcal{D}$, the density of the killed Brownian motion on the wedge $ \mathcal{D} $ is given by
	\begin{align}
		\label{formula_stopped}
		& \mathbb{P}^{x}(W_t \in dy, \ \tau > t) = \frac{2r}{t \alpha} e^{-\frac{r^2 + r_0^2}{2t}} \sum_{n=1}^{\infty} I_{n\pi / \alpha}\left(\frac{rr_0}{t}\right) \sin\left(\frac{n\pi \theta}{\alpha}\right) \sin\left(\frac{n \pi \theta_0}{\alpha}\right) dr d\theta .
	\end{align}
\end{theorem}

\section{Analytic formulas for the density of the reflected process }
\label{section:density_formulas}
\label{section_alpha_pi_m}

In this section, we give explicit expressions for the densities of the reflected Brownian motion.
Before doing this, we review the one dimensional case which will also help explain the motivation for the simulation methods to be introduced later.
The one-dimensional case can be treated using the classic reflection principle, which directly gives the density of the killed and reflected Brownian motions.
In the one-dimensional case, we can assume that $\mathcal{D} = [0,+\infty)$ and that $\Sigma = 1$. Then for all non-negative measurable functions $f : \mathcal{D} \rightarrow \mathbb{R}$ and $x_0 \ge 0$,
\begin{align}
\label{one_dimension_killed}
& \mathbb{E}^{x_0}\left[f(W_T)\mathds{1}_{\tau > T}\right] = \mathbb{E}^{x_0}\left[ f(W_T) \mathds{1}_{W_T > 0} \left( 1 - e^{-\frac{x_0(x_0+W_T)}{T}} \right) \right], \\
\label{one_dimension_reflected}
& \mathbb{E}^{x_0}\left[f(X_T)\right] = \mathbb{E}^{x_0}\left[ f(W_T) \mathds{1}_{W_T > 0} \left( 1 + e^{-\frac{x_0(x_0+W_T)}{T}} \right) \right].
\end{align}
The above formulas show that the simulation of the killed and reflected Brownian motions in the one-dimensional case can be easily performed using changes of measures on the original unrestricted Brownian motion.

\bigskip

 The proof in the reflected case in $\mathbb{R}^2$ shares common arguments with the killed case. 
Throughout these arguments, one uses results from classical theory of partial differential equations on wedges. For a general reference on this topic, we refer the reader to \cite[chapter IV]{Lady} and \cite{kou_zhong2016} where one can find existence and uniqueness results for the partial differential equations that are related to the problems  in this article.

\begin{definition}
	\label{def:1}
	We define the normal reflection process of a continuous adapted stochastic process $\xi$ on the subset $\mathcal{D}$, as the unique solution $ (X_t,L_t)_{t\geq 0} $ of the following problem (see for example \cite{pilipenko2014}):
	\begin{enumerate}
		\item $X_t = \xi_t + L_t$, where $L=(L_t)_{t\geq 0}$ is some adapted process.
		\item $\forall t \ge 0, \ X_t \in \mathcal{D}$.
		\item $L$ is a  continuous process of bounded variation and $L_0 = 0$.
		\item $\forall t \ge 0, \ L_t=\int_0^t n(X_s)\mathds{1}_{X_s \in \partial \mathcal{D}}d|L|_s $, which means that the process $L$ increases only when $X_t \in \partial \mathcal{D}$ and that the reflection is normal to the boundary.
	\end{enumerate}
\end{definition}
Therefore the normally reflected Brownian motion is the process that is obtained from the above equation by taking $ \xi=W $. In the one dimensional case, $ L $ corresponds to the local time of Brownian motion. 

The existence and uniqueness of a Brownian motion reflected in a wedge, which is a non-smooth domain, has been proved in the decorrelated case in \cite{varadhan1985}.
In this article we will only consider normal reflections in the case of unitary diffusion matrix. Note that if we use the decorrelation step as in \eqref{eq:a_prime} for the reflected Brownian motion, this transformation changes the normal reflected process into an obliquely reflected process. We do not know how to obtain the density of such a process.

\begin{theorem}[Reflected case]
\label{Rformula_alpha_pi_m}
Assume that $\mathcal{D}=\langle \pi/m \rangle $ for some $m \in \mathbb{N}$. Then we have the following formulas for the density of the reflected Brownian motion on the wedge $ \mathcal{D}=\langle \pi/m \rangle $:
\begin{align}
\label{alpha_pi_m_reflected}
\forall t>0, \ x, y \in \mathcal{D}, \ \mathbb{P}^{x}(X_t \in dy) & = \frac{1}{2\pi t} \sum_{k=0}^{2m-1} e^{-\frac{|x-T_ky|^2}{2t}} dy .
\end{align}
In the general case for $\alpha \in (0,2\pi)$, $ \mathcal{D}=\langle \alpha\rangle$ and $x = (r_0 \cos(\theta_0),r_0 \sin(\theta_0))\in \mathcal{D}$ we have that for any $y=(r\cos(\theta),r\sin(\theta))\in \mathcal{D}$, the density of the reflected Brownian motion on the wedge $ \mathcal{D}=\langle \alpha \rangle $ is given by
\begin{align}
\label{formula_reflected}
& \mathbb{P}^{x}(X_t \in dy) = \frac{2r}{t \alpha} e^{-\frac{r^2 + r_0^2}{2t}} \left( \frac{1}{2}I_0\left(\frac{rr_0}{t}\right) + \sum_{n=1}^{\infty} I_{n\pi / \alpha}\left(\frac{rr_0}{t}\right) \cos\left(\frac{n\pi \theta}{\alpha}\right) \cos\left(\frac{n \pi \theta_0}{\alpha}\right) \right) dr d\theta . 
\end{align}
\end{theorem}
The proof of the above theorem is given in \ref{appendix:B}.
The idea of using the simple case $ \mathcal{D}=\langle \pi/m \rangle  $ to infer a general result appears in the literature of  the image method. It has been used in a non-trivial way in order to deduce stationary measures for reflected processes (see e.g. \cite{DM}). A formula similar to \eqref{formula_reflected} appears in \cite[page 379 (8)]{carslaw1959}, which is proved using Laplace transform inversion methods. We have decided to include a proof which uses the image method because it gives some intuition on the simulations methods to be introduced later. 

\section{Exact simulation algorithms}
\label{section:algorithms}

Theorems \ref{Kformula_alpha_pi_m} and \ref{Rformula_alpha_pi_m} give formulas that can be directly used to simulate the final values of the reflected and stopped processes in the case that the wedge angle is $\alpha = \pi/m$ (see Sections \ref{algorithm_for_stopped} and \ref{algorithm_for_reflected}). However, direct algorithms that arise due to these results demand the use of approximations and  larger computational time in general. Instead, we investigate an alternative simulation method using the tractable case $\alpha = \pi/m$ and give algorithms for the exact simulation of the reflected and stopped processes in any wedge.

\subsection{Formulas for the simulation of $(\tau,W_\tau)$ in the case $\alpha = \pi/m$}


Before providing the simulation method, we give a formula that will be the key to provide a simplified simulation method which avoids the calculation of Bessel functions.

\begin{theorem}
Let $\mathcal{V}=\langle \alpha^-,\alpha^+\rangle$ and assume that  $\alpha^+ - \alpha^-  = \pi/m=: \alpha$ for some $m \in \mathbb{N}$ and let $\tau$ be the hitting  time on the wedge $\mathcal{V}$ for $W$  which starts at  $x_0= (r_0 \cos(\theta_0), r_0 \sin(\theta_0)) \in \mathcal{V}$, then
\begin{equation}
\mathbb{P}^{x_0}(\tau \in dt, \ W_\tau \in dy^\pm) = \frac{r_0}{2\pi t^2}e^{-\frac{r^2+r_0^2}{2t}} \sum_{k=0}^{m-1} \sin\left(\gamma_k^\pm\right)e^{\frac{rr_0}{t}\cos\left(\gamma_k^{\pm}\right)} dr dt,
\end{equation}
where $y^\pm=(r\cos(\alpha^\pm),r\sin(\alpha^\pm)) \in \partial \mathcal{V}^{\pm}$. We use the notation
\begin{align}
\gamma_k^\pm :=& \ \pm \alpha^\pm \pm \frac{2k\pi}{m} \mp \theta_0.
\end{align}
\end{theorem}

\begin{proof}
We can assume that $\alpha^-=0$ without loss of generality. We use the formulas \cite[(1.5), (1.6)]{metzler2010}. If $y \in \partial \mathcal{D}^-$, writing $y = (r,0)$ for some $r > 0$, we have:
\begin{equation}
\label{joint_density_tau_r_minus}
\mathbb{P}^{x_0}(\tau \in dt, \ W_\tau \in dy) = \frac{\pi}{\alpha^2 t r} e^{-\frac{r^2 +r_0^2}{2t}} \sum_{n=1}^\infty n \sin\left( \frac{n\pi\theta_0}{\alpha}\right) I_{n\pi/\alpha}\left(\frac{rr_0}{t}\right) dr dt .
\end{equation}

Using \eqref{formula_stopped} and switching the order of derivation and integration, which can be justified (see the proof of Theorem \ref{Rformula_alpha_pi_m} with Corollary \ref{cor:2} in the Appendix), we obtain for $y \in \partial \mathcal{D}^-$:
$$ \mathbb{P}^{x_0}(\tau \in dt, \ W_\tau \in dy) = \frac{1}{2r^2} \left.\frac{\partial}{\partial \theta} \right|_{\theta = 0} \mathbb{P}^{x_0}(\tau > t, W_t \in d(r\cos(\theta),r\sin(\theta))).$$
So that in the case $\alpha = \pi/m$, using \eqref{alpha_pi_m_killed} in polar coordinates and the definition of $ \vartheta_k $:
\begin{align}
\mathbb{P}^{x_0}(\tau \in dt, \ W_\tau \in dy) & = \frac{1}{2r^2} \left.\frac{\partial}{\partial \theta} \right|_{\theta = 0} \frac{r}{2\pi t} e^{-\frac{r^2+r_0^2}{2t}} \sum_{k=0}^{2m-1} (-1)^k e^{\frac{1}{t}rr_0\cos(\theta_0 - \vartheta_k )} dr dt \notag\\
& = \frac{r_0}{2\pi t^2} e^{-\frac{r^2+r_0^2}{2t}} \sum_{k=0}^{m-1} \sin(\theta_0 - 2k\alpha) e^{\frac{rr_0}{t}\cos(\theta_0 - 2k\alpha)} dr dt.
\label{eq:5.4}
\end{align}
Likewise, we obtain a similar formula for the case $y \in \partial \mathcal{D}^+$.
\end{proof}

From here, we present a method to simulate $\tau$ and $W_\tau$ in the case of the wedge $\mathcal{V}=\langle \alpha^-,\alpha^+\rangle$ with $\alpha^+-\alpha^- = \pi/m$. For the simulation of $W_\tau$, we use the same method which is proposed in \cite[Section 2.1]{metzler2008} but we restrict to the case $\alpha=\pi/m$. On the other hand, we propose an unbiased simulation method for  $\tau$ in comparison with the method proposed in \cite{metzler2008} which has a bias difficult to estimate (see \cite[Proposition 2.1.8]{metzler2008} and the remark that follows).

\textbf{Simulation of $W_\tau$:} First, we simulate the selection of the boundary line on which $W_\tau$ arrives, with a Bernoulli distribution. These two boundaries are denoted by $ \partial \mathcal{V}^\pm=
\lbrace (r\cos(\alpha^\pm),r\sin(\alpha^\pm)): \ r \ge 0 \rbrace$. Following \cite[Corollary 2.1.5]{metzler2008} we have:
\begin{equation}
\label{eq:which_frontier}
\mathbb{P}^{x_0}(W_\tau \in \mathcal{V}^+) = \frac{\theta_0-\alpha^-}{\alpha^+-\alpha^-} ,
\end{equation}

We simulate on which frontier $W_\tau$ arrives first and then we directly simulate $W_\tau$, without simulating $\tau$ yet.
Following \cite[Proposition 2.1.6]{metzler2008} one can compute explicit formulas for the distribution function of the radius of exit given that it exits at $ \mathcal{V}^+  $ or $ \mathcal{V}^-  $. Also its inverse can be computed and therefore the inverse transformation method for simulation can be applied. Using these formulas, we can simulate the radius $r_\tau$ as:
\begin{equation}
\label{eq:metzler_radius}
r_\tau = \left\lbrace \begin{array}{ll}
r_0 \left( \cos\left(\frac{\pi \theta_0}{\alpha}\right) - \frac{\sin\left(\frac{\pi \theta_0}{\alpha}\right)}{\tan((\pi - \pi\theta_0/\alpha)(U-1))} \right)^{\alpha/\pi} \ & \text{ if } W_\tau \in \mathcal{V}^- , \\
r_0 \left(-\cos\left(\frac{\pi \theta_0}{\alpha}\right) - \frac{\sin\left(\frac{\pi \theta_0}{\alpha}\right)}{\tan((\pi\theta_0/\alpha)(U-1))} \right)^{\alpha/\pi} \ & \text{ if } W_\tau \in \mathcal{V}^+ ,
\end{array} \right.
\end{equation}
with $U \sim \mathcal{U}([0,1])$.

\smallskip

\textbf{Simulation of $\tau$:} Knowing $W_\tau$, we simulate $\tau$ according to the conditional formula (see \eqref{eq:5.4} and recall that $ dy^{\pm}=(dr\cos(\alpha^\pm),dr\sin(\alpha^\pm)) $):
\begin{equation}
\label{simulation_tau}
\mathbb{P}^{x_0}(\tau \in dt | \ W_\tau =y^\pm) = \left((\mathbb{P}^{x_0}(W_\tau \in dy^\pm))^{-1} dr\right) \frac{r_0}{2\pi t^2} \sum_{k=0}^{m-1} \sin(\gamma_k^\pm)\exp\left(-\frac{(r-r_0\cos(\gamma_k^\pm))^2 + r_0^2\sin^2(\gamma_k^\pm)}{2t}\right) dt.
\end{equation}
Since some of the coefficients $\sin(\gamma_k^\pm)$ are negative, we cannot directly simulate according to this distribution as a mixture. Instead we select the indexes $k_0^\pm, \ \ldots, \ k_{p^\pm}^\pm$ such that for all $i$, $\sin(\gamma_{k_i^\pm}^\pm) \ge 0$ and write:
\begin{equation}
\label{eq:tau_sampling}
\mathbb{P}^{x_0}(\tau \in dt | \ W_\tau =y^{\pm}) \le \left((\mathbb{P}^{x_0}(W_\tau \in dy^\pm))^{-1}dr\right) \frac{r_0}{2\pi t^2} \sum_{i=0}^{p^\pm}\sin(\gamma_{k_i^\pm}^\pm) \exp\left(-\frac{(r-r_0\cos(\gamma_{k_i^\pm}^\pm))^2 + r_0^2\sin^2(\gamma_{k_i^\pm}^\pm)}{2t}\right) dt.
\end{equation}
This is the distribution of a discrete mixture of random variables $\xi_k$ for $0 \le k \le p^\pm$, where $\xi_k$ is the inverse of an exponential random variable of parameter $((r-r_0\cos(\gamma_{k_i^\pm}^\pm))^2 + r_0^2\sin^2(\gamma_{k_i^\pm}^\pm))/2 $.
With this information, we build an acceptance-rejection sampling, based on the density on the right hand side of \eqref{eq:tau_sampling} and accept the sample value $t$ under the condition
\begin{align*}
U & < \left(\sum_{k=0}^{m-1} \sin(\gamma_k^\pm) \exp\left(-\frac{(r-r_0\cos(\gamma_k^\pm))^2 + r_0^2\sin^2(\gamma_k^\pm)}{2t}\right)\right) \\
& \quad \times \left(\sum_{i=0}^{p^\pm}\sin(\gamma_{k_i^\pm}^\pm) \exp\left(-\frac{(r-r_0\cos(\gamma_{k_i^\pm}^\pm))^2 + r_0^2\sin^2(\gamma_{k_i^\pm}^\pm)}{2t}\right)\right)^{-1} ,
\end{align*}
with $U \sim \mathcal{U}([0,1])$.

\smallskip

\textbf{Simulation in the case $ W $ does not exit the wedge:} So far, we have provided a methodology in order to simulate the exit location and time for the simpler wedge $ \mathcal{V}=\langle \alpha^-,\alpha^+\rangle $ with $ \alpha^+-\alpha^-=\pi/m $. Now we discuss the case when $ W $ does not leave this wedge in the time interval $ [0,T] $. To simulate the final value $W_T$ of the Brownian motion starting from $x_0$ conditionally to the fact that it does not exit the wedge $\mathcal{V}$ on $[0,T]$, we need to simulate according to the density \eqref{alpha_pi_m_killed}, which is a sum where some terms are negative. We propose the following method.

A slight generalization of \eqref{alpha_pi_m_killed} to the wedge $\mathcal{V}=\langle \alpha^-,\alpha^+\rangle$ with
$$ \widetilde{\vartheta}_k := \left\lbrace\begin{array}{ll}
\vartheta_k + 2\alpha^- \ & \text{ if } k \text{ odd} \\
\vartheta_k \ & \text{ if } k \text{ even},
\end{array} \right.$$
gives, for $y \in \mathcal{V}$:
\begin{equation}
\label{eq:alpha_pi_m_killed_simulation}
\mathbb{P}^{x_0}(W_T \in dy| \ \tau >T) = \frac{\mathbb{P}^{x_0}(\tau>T)^{-1}}{2\pi T} \sum_{k=0}^{2m-1} (-1)^k e^{-\frac{|x_0-\widetilde{T}_ky|^2}{2T}}dy \le \frac{\mathbb{P}^{x_0}(\tau>T)^{-1}}{2\pi T} \sum_{k=0}^{m-1} e^{-\frac{|x_0-\widetilde{T}_{2k}y|^2}{2T}}dy ,
\end{equation}
where $\widetilde{T}_k(r\cos(\theta),r\sin(\theta)) = (r\cos(\widetilde{\vartheta}_k),r\sin(\widetilde{\vartheta}_k))$. We then apply acceptance-rejection sampling with the reference density proportional to $\frac{1}{2\pi T} \sum_{k=0}^{m-1} e^{-\frac{|x_0-\widetilde{T}_{2k}y|^2}{2T}}$ and accept the sample $y$ under the condition
\begin{equation}
U < \left(\sum_{k=0}^{2m-1} (-1)^k e^{-\frac{|x_0-\widetilde{T}_{k}y|^2}{2T}}\mathds{1}_{y \in \mathcal{V}}dy \right) \left(\sum_{k=0}^{m-1} e^{-\frac{|x_0-\widetilde{T}_{2k}y|^2}{2T}} \right)^{-1} ,
\end{equation}
with $U \sim \mathcal{U}([0,1])$.

\subsection{Algorithm for the simulation of the stopped Brownian motion: General case}
\label{algorithm_for_stopped}

In this subsection, we give a recursive algorithm to simulate the final value $W_{T\wedge \tau}$ of the stopped process on the wedge $\mathcal{D}$, for any angle $\alpha \in (0,2\pi)$. This algorithm will be used for the simulation algorithm of the reflected Brownian motion, but it also has its own interest as it can be used as itself for the simulation of the stopped Brownian motion. In what follows, we denote $\mathcal{F}_n$ the sigma algebra generated by the algorithm until the step $n$, i.e. $\mathcal{F}_n = \sigma( \tau_1, y_1, \ldots, \tau_n, y_n)$, where $\tau_1, \ \ldots , \ \tau_n$ and $y_1, \ \ldots, \ y_n$ are random variables generated by the algorithm. We also denote by $\mathbb{P}_n$ and $\mathbb{E}_n$ the conditional probability and expectation with respect to $\mathcal{F}_n$.

First, we propose an algorithm for the case where the angle is $\pi/m$. At the step $n$ of the algorithm we will define a wedge $\mathcal{V}_n \subset \mathcal{D}$ of angle $\pi/m$ for some $m \in \mathbb{N}$ and then simulate the process stopped on this wedge. Define $\Theta := \max \left( \{ \pi/m, \ m \in \mathbb{N} \} \cap (0, \alpha] \right)$. Note that since $\alpha \in (0,2\pi)$, we have $\alpha/2 < \Theta \le \alpha $.

\medskip

\underline{\it Algorithm I:}
Start at the point $x_0 = (r_0\cos(\theta_0),r_0\sin(\theta_0)) =: y_0$ at time $T_0 = 0$. The algorithm follows these steps for $n\in\mathbb{N}$:
\begin{enumerate}
	\item Given $ y_n = (r_n\cos(\theta_n),r_n\sin(\theta_n))\in \mathcal{V}_n $, define the angles:
\begin{align*}
	&\beta_{n+1}^- :=
	\left\lbrace \begin{array}{cc}
		0 & \text{ if } \theta_n \in \left[ 0, \frac{\Theta}{2} \right] \\
		\theta_n - \frac{\Theta}{2} & \text{ if } \theta_n \in \left[\frac{\Theta}{2}, \alpha - \frac{\Theta}{2} \right] \\
		\alpha - \Theta & \text{ if } \theta_n \in \left[ \alpha - \frac{\Theta}{2}, \alpha \right]
	\end{array} \right. \\
	&\beta_{n+1}^+ :=  \beta_{n+1}^- + \Theta.
\end{align*}
	By the definition of $ \Theta $, we have  $\Theta \le \alpha$ and therefore $0 \le \Theta/2 \le \alpha - \Theta/2 \le \alpha$.

	\item Consider the wedge $\mathcal{V}_{n+1}=\langle \beta_{n+1}^-,\beta_{n+1}^+\rangle $ of angle $\Theta$, which satisfies $\mathcal{V}_{n+1} \subset \mathcal{D}$ and $y_n=(r_n\cos(\theta_n),r_n\sin(\theta_n))$ $\in \mathcal{V}_{n+1}$.

	\item Simulate on which one of the two sets $\partial \mathcal{V}_{n+1}^+$ or $\partial \mathcal{V}_{n+1}^-$, the process $(W_t)_{t \ge T_n}$ starting from $y_n$, at time $T_n$, exits the wedge  $ \mathcal{V}_{n+1}$ and define $\theta_{n+1} := \beta_{n+1}^\pm$ according to the simulation result. This is done using a Bernoulli variable and the formula \eqref{eq:which_frontier}.
%
	\item Simulate $y_{n+1} \in \partial \mathcal{V}_{n+1}$, which is the value of $W$ when it reaches $\partial \mathcal{V}_{n+1}$ for the first time after starting at $y_n$ at time $T_n$, and knowing which one of the two events $y_{n+1}\in \partial \mathcal{V}_{n+1}^\pm$ has occurred, using formula \eqref{eq:metzler_radius}. Denote the exit point as
	$$ y_{n+1} = (r_{n+1}\cos(\theta_{n+1}), r_{n+1}\sin(\theta_{n+1})) .$$

	\item Simulate $\tau_{n+1}$, the time for the process $W$ starting at $y_n=(r_n\cos(\theta_n),r_n\sin(\theta_n))$ at time $T_n$ to reach $\partial \mathcal{V}_{n+1}$, knowing the event $W_{T_n+\tau_{n+1}} = y_{n+1}$. This can be done using the formula \eqref{simulation_tau} and the acceptance-rejection procedure described after it. Define
	$$ T_{n+1} := T_n + \tau_{n+1},$$
	so that $ W_{T_{n+1}} = y_{n+1}$.
\end{enumerate}

If $T_n + \tau_{n+1} < T$ then the algorithm iterates. This algorithm stops under one of these two conditions:
\begin{itemize}
	\item Condition 1:  $T_n + \tau_{n+1} \geq  T$. Then simulate the final value $W_T$ knowing that $W_{T_n} = y_n$ and that for $t\in [T_n, T]$, $W_t \in \mathcal{V}_n$. We do this simulation conditionally to the event $T_n + \tau_{n+1} > T$ and ``forget`` the exact simulated value of $\tau_{n+1}$ as well as the value of $W_{T_n+\tau_{n+1}}$. This is justified in Proposition \ref{prop:forget_step}, using $X := (W_{T \wedge T_{n+1}}, T_{n+1}\mathds{1}_{T_{n+1}<T})_{|\mathscr{F}_n}$, $Y := (W_{T_{n+1}}, T_{n+1})_{|\mathscr{F}_n}$ and $A = \mathbb{R}^2 \times [0,T]$. We simulate the value of  $W_T$ using \eqref{eq:alpha_pi_m_killed_simulation} and the acceptance-rejection procedure described after it.

	\item Condition 2:  $\theta_{n+1} = 0$ or $\theta_{n+1} = \alpha$. Then we obtain $\tau:=T_{n+1} < T$, i.e. the Brownian motion reaches the wedge $\mathcal{D}$ before the time $T$ and the result of the simulation is
$$ W_{T \wedge \tau} = (r_{n+1}\cos(\theta_{n+1}),r_{n+1}\sin(\theta_{n+1})) .$$
\end{itemize}
An illustration of this algorithm is given in Figure \ref{fig:algorithm_stopped}. Next, we prove that this algorithm can be carried out in a finite number of steps. 


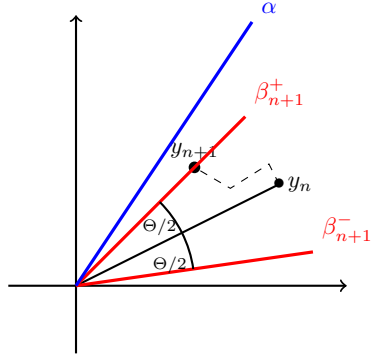
\begin{figure}
\centering
\begin{tikzpicture}
	[thick,scale=0.9, every node/.style={scale=0.9}]
\coordinate (beta_minus) at (3.5,0.5);
\coordinate (beta_plus) at (2.5,2.5);
\coordinate (alpha) at (2.6,3.9);
\coordinate (z_n) at ($(beta_minus)!0.5!(beta_plus)$);
\coordinate (z_nn) at ($(0,0)!0.7!(beta_plus)$);
\coordinate (origin) at (0,0);

\draw[thick,->] (0,-1) -- (0,4);
\draw[thick,->] (-1,0) -- (4,0);

\filldraw (z_nn) circle (2pt) node[anchor = south] {$y_{n+1}$};
\node at (z_n) {\textbullet};
\draw (0,0) -- (z_n) node[anchor = west] {$y_n$};
\draw[very thick,red] (0,0) -- (beta_minus) node[anchor=south west] {$\beta_{n+1}^-$};
\draw[very thick,red] (0,0) -- (beta_plus) node[anchor=south west] {$\beta_{n+1}^+$};
\draw[very thick,blue] (0,0) -- (alpha) node[anchor =south west] {$\alpha$};

\coordinate (w_1) at ($(z_n)!0.3!(beta_plus)$);
\coordinate (w_2) at ($(w_1)!0.2!(origin)$);
\draw[thin,dashed] (z_n) -- (w_1) -- (w_2) -- (z_nn);

\draw pic[draw,angle radius=1.75cm,"${\scriptstyle\Theta/2}$" anchor=west] {angle = beta_minus--origin--z_n};
\draw pic[draw,angle radius=1.75cm,"${\scriptstyle\Theta/2}$" anchor=south west] {angle = z_n--origin--beta_plus};

\end{tikzpicture}

\caption{Example of the domains of simulation at step $n+1$.}
\label{fig:algorithm_stopped}
\end{figure}

\smallskip

\begin{proposition}
\label{stopped_finite_time}
Algorithm I ends in finite time, i.e. the number of required iterations to finish the algorithm, denoted by $N$, is such that $N < \infty$ almost surely. More precisely, for all $K \in \mathbb{N}$, $\mathbb{P}^{x_0}(N \ge K)  \le 2^{-\left\lfloor K/2 \right\rfloor} $. 
\end{proposition}
\begin{proof}
Fix $K \in \mathbb{N}$. For a wedge $\mathcal{V}_n$, $n \le N$ in Algorithm I, we say that ``$\mathcal{V}_n$ intersects the boundary of $ \mathcal{D} $" if $\beta_{n}^+ = \alpha$ or $\beta_{n}^- = 0$. For $n \le N$, we have two possible cases:
\begin{itemize}
	\item $\mathcal{V}_n$ does not intersect the boundary of $ \mathcal{D} $. Then, using the fact that $\Theta > \frac{\alpha}{2}$ and the definition of $\beta_{n+1}^\pm$, $\mathcal{V}_{n+1}$ intersects the boundary of $ \mathcal{D} $, and $\theta_n$ is closer to $\beta_{n+1}^\pm$ than $\beta_{n+1}^\mp$, where here the sign $\pm$ is such that $\partial \mathcal{V}_{n+1}^\pm \cap \partial \mathcal{D} \ne \lbrace 0 \rbrace$. So using \eqref{eq:which_frontier}, $ \mathbb{P}_n(y_{n+1} \notin \partial \mathcal{D}) \le 1/2$.
	\item $\mathcal{V}_n$ intersects the boundary of $ \mathcal{D} $. Then if $\theta_{n+1} \notin \lbrace 0, \alpha \rbrace$, $\mathcal{V}_{n+1}$ does not intersect the boundary of $ \mathcal{D} $ and by the same reasoning as above $\mathbb{P}_n(y_{n+2} \notin \partial \mathcal{D}) \le 1/2$.
\end{itemize}
Then, by tower property of conditional expectations,  we have:
$$ \mathbb{P}^{y_0}(N \ge K) \le \mathbb{P}^{y_0}(y_1 \notin \partial \mathcal{D},\ldots,y_K \notin \partial \mathcal{D}) \le 2^{-\left\lfloor \frac{K}{2} \right\rfloor} .$$
So that:
$$ \mathbb{P}^{y_0}(N = \infty) = \underset{K \rightarrow \infty}{\lim} \mathbb{P}^{y_0}(N \ge K) = 0 .$$
\end{proof}


\subsection{Algorithm for the simulation of the reflected Brownian motion}
\label{algorithm_for_reflected}

We now state an algorithm to simulate the final value of the reflected Brownian motion $X_T$ in a wedge $\mathcal{D}=\langle\alpha \rangle$ of any angle $\alpha \in (0,2\pi)$.

\underline{\it Algorithm II:}
Choose $\Theta := \max \left( \{ \pi/m, \ m \in \mathbb{N} \} \cap (0, \alpha] \right)$ as before. Start at the point $x_0 = (r_0 \cos(\theta_0), r_0 \sin(\theta_0))$ $=: y_0$ at time $T_0 = 0$. In general, suppose that for $n\in\mathbb{N}$ the point $ y_n=(r_n\cos(\theta_n), r_n \sin(\theta_n)) $ has already been simulated.

\begin{enumerate}

	\item Define the angles:
	\begin{align*}
	\beta_{n+1}^- & := \theta_n - \frac{\Theta}{2}, \\
	\beta_{n+1}^+ & := \theta_n + \frac{\Theta}{2} = \beta_{n+1}^- + \Theta .
	\end{align*}

	\item Consider the wedge $\mathcal{V}_{n+1}:=\langle \beta_{n+1}^-, \beta_{n+1}^+\rangle$. Note that although $y_n \in \mathcal{V}_{n+1}$, we do not necessarily have that $\mathcal{V}_{n+1} \subset \mathcal{D}$.

	\item Simulate the random variables $\tau_{n+1}$ and $z_{n+1}$, which are respectively the time and the final point of a Brownian motion $Z=(Z_t)_{t\geq 0}$ starting at $y_n$ at time $0$ and stopped on the wedge $\mathcal{V}_{n+1}$ using the steps 3-5 of Algorithm I in Section \ref{algorithm_for_stopped}. Note that, since $\theta_n$ is in the middle of the wedge $\mathcal{V}_{n+1}$, we have by symmetry $ \mathbb{P}^{z_n}(Z_{\tau_{n+1}} \in \partial \mathcal{V}_{n+1}^\pm) = 1/2$.

	\item If $T_n + \tau_{n+1} < T$, then we define $r_{n+1}$ and $\widetilde{\theta}_{n+1}$ to be respectively the radius and the angle of $z_{n+1}$. Since it is possible that $z_{n+1} \notin \mathcal{D}$, we define $\theta_{n+1}$ as follows:
	$$ \theta_{n+1} = \left\lbrace \begin{array}{ll}
		\widetilde{\theta}_{n+1} & \text{ if } \widetilde{\theta}_{n+1} \in [0,\alpha], \\
		- \widetilde{\theta}_{n+1} & \text{ if } \widetilde{\theta}_{n+1} = \beta_{n+1}^- \text{ and } \beta_{n+1}^- < 0, \\
		2\alpha - \widetilde{\theta}_{n+1} & \text{ if } \widetilde{\theta}_{n+1} = \beta_{n+1}^+ \text{ and } \beta_{n+1}^+ > \alpha.
	\end{array} \right.$$
	And then we define $y_{n+1} := (r_{n+1} \cos(\theta_{n+1}), r_{n+1} \sin(\theta_{n+1}))$, so that $y_{n+1} \in \mathcal{D}$. Note that if $z_{n+1} \notin \mathcal{D}$, then $y_{n+1}$ is the reflection of $z_{n+1}$ with respect to the line $\lbrace \theta = 0 \rbrace$ in the second case or $\lbrace \theta = \alpha \rbrace$ in the third case in the definition of $ \theta_{n+1} $. Then define the time:
	$$ T_{n+1} = T_n + \tau_{n+1}$$
and the algorithm iterates.

	\item If $T_n + \tau_{n+1} > T$, we simulate $z_{n+1} \in \mathcal{V}_{n+1}$ as the value at time $T - T_n$ of a Brownian motion starting at $y_n$ and conditionally to the fact that it stays in the wedge $\mathcal{V}_{n+1}$ in the time interval $[0,T-T_n]$. This can be done using the acceptance-rejection method given in \eqref{eq:alpha_pi_m_killed_simulation} and what follows based on Proposition \ref{prop:forget_step} as it was done in the termination Condition 1 in Algorithm I in Section \ref{algorithm_for_stopped}. Then we define $y_{n+1}$ in function of $z_{n+1}$ as in step 4.
	Finally, we stop the algorithm, and the resulting value of the simulation is $y_{n+1}$.
\end{enumerate}
An illustration of this algorithm is given in Figure \ref{fig:algorithm_reflected}.


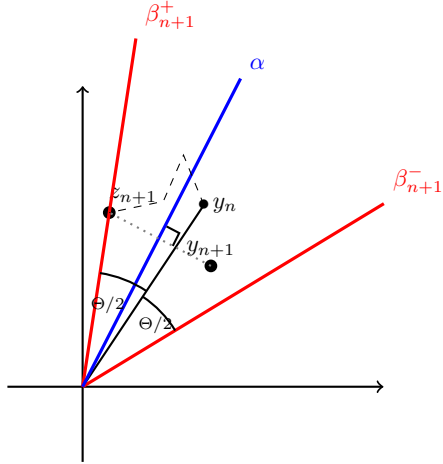
\begin{figure}
\centering
\begin{tikzpicture}[thick,scale=1., every node/.style={scale=0.9}]
\def \theta{50}
\coordinate (origin) at (0,0);
\coordinate (alpha) at (2.1,4.1);
\coordinate (theta_n) at (2.6,3.9);
\coordinate (z_n) at ($(0,0)!0.62!(theta_n)$);
\coordinate (beta_minus) at ([rotate around={-\theta*0.5:(origin)}]theta_n);
\coordinate (beta_plus) at ([rotate around={\theta*0.5:(origin)}]theta_n);
\coordinate (tilde_z_nn) at ($(origin)!0.5!(beta_plus)$);
\coordinate (z_nn) at ([rotate around={-38:(origin)}]tilde_z_nn);
\coordinate (intersection) at ($(tilde_z_nn)!0.5!(z_nn)$);

\draw[thick,->] (0,-1) -- (0,4);
\draw[thick,->] (-1,0) -- (4,0);

\filldraw (tilde_z_nn) circle (2pt) node[anchor = south] {\hspace{0.7cm}$z_{n+1}$};
\filldraw (z_nn) circle (2pt) node[anchor = south] {$y_{n+1}$};
\node at (z_n) {\textbullet};
\draw (0,0) -- (z_n) node[anchor = west] {$y_n$};
\draw[very thick,red] (0,0) -- (beta_minus) node[anchor=south west] {$\beta_{n+1}^-$};
\draw[very thick,red] (0,0) -- (beta_plus) node[anchor=south west] {$\beta_{n+1}^+$};
\draw[very thick,blue] (0,0) -- (alpha) node[anchor =south west] {$\alpha$};

\coordinate (w_1) at ($(z_n)!0.3!(beta_plus)$);
\coordinate (w_2) at ($(w_1)!0.2!(origin)$);
\draw[thin,dashed] (z_n) -- (w_1) -- (w_2) -- (tilde_z_nn);

\draw pic[draw,angle radius=1.6cm,"${\scriptstyle\Theta/2}$" anchor=south west] {angle = beta_minus--origin--z_n};
\draw pic[draw,angle radius=1.7cm,"${\scriptstyle\Theta/2}$" anchor=south] {angle = z_n--origin--beta_plus};

\draw[gray, dotted] (z_nn) -- (tilde_z_nn);
\draw ($(intersection)!0.25!(z_nn)$) -- +($(intersection)!0.08!(alpha)-(intersection)$) -- ($(intersection)!0.08!(alpha)$);

\end{tikzpicture}

\caption{Example of domains of simulation at the step $n+1$ in the reflected case. }
\label{fig:algorithm_reflected}
\end{figure}

\begin{proposition}
\label{finite_time_reflected}
Define $ N := \inf \lbrace n \in \mathbb{N}, \ \tau_1 + \ldots + \tau_n > T \rbrace$. Then $ N<\infty $ almost surely. That is, the algorithm terminates in finite time.
\end{proposition}
\begin{proof}
The stopping times $(\tau_i)_i$ follow the same law as $(\widetilde{\tau}_i)$, defined by $\widetilde{\tau}_0 = 0$ and:
\begin{align}
\label{eq:definition_N:2}\widetilde{\tau}_{i+1} := \inf \left \lbrace t>0 : \ |\theta(B_{t+\widetilde{\tau_i}})-\theta(B_{\widetilde{\tau}_i})| \ge \frac{\pi}{2m} \right\rbrace,
\end{align}
where $m = \pi/\Theta\in\mathbb{N}$, $B=(B_t)_{t\geq 0}$ is a standard two-dimensional Brownian motion and $\theta(B_t)$ denotes the angle of the process\footnote{We choose the angle of the process $(\theta(B_t))_{t\geq 0}$ so that it is continuous with respect to $t$.}  $B$. Then:
$$ \mathbb{P}^{x_0}\left( \sum_{i=1}^\infty \tau_i < \infty \right) = \mathbb{P}^{x_0}\left( \sum_{i=1}^\infty \widetilde{\tau}_i < \infty \right).$$

But, if $\mathcal{T} := \sum_{i=1}^\infty \widetilde{\tau}_i < \infty$, then there exist two random sequences, $(t^1_n)$ and $(t^2_n)$, increasing and converging to $\mathcal{T}$, such that $\theta(B_{t^i_n}) = k_n^i \pi/(2m)+\theta_0$ for all $n \in \mathbb{N}$ and $i \in \lbrace 1,2 \rbrace$, and where $k_n^i \in \mathbb{Z}$ has the same parity as $i$, so that the difference between $\theta(B_{t_n^1})$ and $\theta(B_{t_{n'}^1})$ is at least $\pi/(2m)$ for all $ n,n'\in\mathbb{N} $. Taking $n \rightarrow \infty$, we have necessarily that $B_\mathcal{T}=0$, which occurs with probability $0$ on any closed time interval.
\end{proof}

\section{Folding number in a wedge and complexity 
}
\label{section:complexity}


In this section we study the complexity of Algorithms I and II in Sections \ref{algorithm_for_stopped} and \ref{algorithm_for_reflected} in separate cases. Recall that $N$ denotes the number of iterations that an algorithm requires to finish one simulation, i.e.
\begin{align*}
N := \inf \lbrace n \in \mathbb{N} : \ \tau_1 + \ldots + \tau_n > T \rbrace.
\end{align*}
For Algorithm I, $N$ is the number of times we touch the boundaries of the corresponding sets $ \mathcal{V}_n $, $ n\in\mathbb{N} $, before touching one of the boundaries of $ \mathcal{D} $ or reaching time $ T $.
For Algorithm II, $N$ is the number of times that a Brownian motion reflected in a wedge $\mathcal{V}$ of angle $\pi/(2m)$ goes from one boundary $\partial\mathcal{V}^\pm$ to another $\partial\mathcal{V}^\mp$ until time $ T $.  
For this reason, we may call $ N $ the number of folds for the simulation of the reflected Brownian motion. 

In this section, we give theoretical properties and bounds for $N$ for each algorithm separately.

\subsection{Majoration of $N$ for the simulation algorithm of the stopped Brownian motion}

For Algorithm I, by Proposition \ref{stopped_finite_time}, we have for all $K \in \mathbb{N}$, $\mathbb{P}^{x_0}(N \ge K) \le 2^{-\left\lfloor K/2 \right\rfloor}$, so
$$\mathbb{E}^{x_0}[N] = \sum_{K=1}^\infty \mathbb{P}^{x_0}(N \ge K) \le \sum_{K=1}^\infty 2^{-\left\lfloor \frac{K}{2} \right\rfloor} = 3.$$
So Algorithm I ends in finite time almost surely and its complexity is finite in expectation.

\subsection{Majoration of the number of folds of the Brownian motion in a wedge}
\label{sec:6.2}

We now investigate the complexity for Algorithm II.
We decompose the standard two-dimensional Brownian motion $W=(W_t)_{t\geq 0}$ in polar coordinates and introduce the notations that will be used in this section and the next one as follows.
The successive stopping times $(\tau_i)_i$ which appear in Algorithm II, have the same law as $(\tilde{\tau}_i)_i$, defined  as in the proof of Proposition \ref{finite_time_reflected}.
Then we have
$$ N\stackrel{\mathscr{L}}{=}\inf \lbrace n \in \mathbb{N} : \ \tilde{\tau}_1 + \ldots + \tilde{\tau}_n > T \rbrace.$$

Moreover, using the skew-product representation of the Brownian motion (see \cite[page 194]{RY}):
 \begin{align}
 \label{eq:sr}
 B_t = R_t U_{F(t)},
 \end{align}
where $(R_t) _{t\geq 0}= (|B_t|)_{t\geq 0}$ is a Bessel process of dimension 2,
 $(U_t)_{t \geq 0}$ is a Brownian motion on $\mathbb{S}^1 \subset \mathbb{R}^2$ and $F$ is defined as
\begin{equation}
F(t) = \int_0^t \frac{ds}{R_s^2},
\end{equation}
which is a strictly increasing process almost surely.
Moreover, the processes $(R_t)_{t \geq 0}$ and $(U_t)_{t \geq 0}$ are independent. If we let $(\theta(U_t))_{t \geq 0}$ be the angle of the process $(U_t)_{t \geq 0}$ then $(\theta(U_t))_{t \geq 0}$ is a standard one-dimensional Brownian motion. 

Define the stopping times $(s_i)_i$ by $s_0 = 0$ and
$$ s_{i+1} = \inf \lbrace t>0 : \ |\theta(U_{t+s_i}) - \theta(U_{s_i})| \ge \pi/(2m) \rbrace.$$
Then, the random variables $s_i$ are i.i.d. and have the law of the time for one-dimensional Brownian motion starting at $0$ to reach the double barrier $\pm \pi/(2m)$. Then using \eqref{eq:definition_N:2} we have that for all $K \in \mathbb{N}$,
$$ \sum_{i=1}^K s_i \overset{\mathscr{L}}{=} F\left( \sum_{i=1}^K \tilde{\tau}_i \right).$$
So the estimation of $ N $ can be simplified as 
 \begin{align}
 \label{eq:nk}
\mathbb{P}^{x_0}(N \ge K) = \mathbb{P}^{x_0}\left(\sum_{i=1}^K \tilde{\tau}_i \le T\right) = \mathbb{P}^{x_0}\left(\sum_{i=1}^K s_i \le F(T)\right).
\end{align}
Note that $F$ is a random change of time, independent of $(s_i)_i$. Furthermore, the Laplace transform of $ s_i $ is explicitly given in \cite[Section 3.5, Proposition 3.5.1.3]{jeanblanc2010} as 
\begin{align*}
	\mathbb{E}[e^{-\lambda s_1}]=\frac{2}{e^{\frac{\sqrt{2}\pi\sqrt{\lambda}}{2m}}+e^{-\frac{\sqrt{2}\pi\sqrt{\lambda}}{2m}}}.
\end{align*} In particular, we have:
$$ (\mathbb{E}[s_i],\text{Var}(s_i)) = \left(\frac{\pi^2}{4m^2},\frac{2}{3}\left(\frac{\pi}{2m}\right)^4 \right)=: (\mu,\sigma^2).$$
\begin{proposition}
	\label{prop:6}
	Let $0 < a < \frac{1}{2}$. Then we have
	$ \mathbb{E}^{x_0}[N^a] <\infty.$ On the other hand, if $ a\geq 1/2 $ then $ \mathbb{E}^{x_0}[N^a] =\infty.$
	%
\end{proposition}
Equation \eqref{eq:nk} together with the estimate in \eqref{tail_bessel_integral} give the result which is a consequence of the fact that the tails of $ F(T) $ are large because the probability that the Bessel process $ R $ is close to zero is large and therefore due to \eqref{eq:sr} the simulation may spend a large time close to the origin. 
\begin{proof}
	We have that for all $K \in  \mathbb{N} $ and $ \delta>1 $:
	\begin{align*}
		\mathbb{P}^{x_0}(N^a \ge K) 
		= &\int_0^{\frac{\mu K^{1/a}}{\delta}} \mathbb{P}^{x_0}\left( \sum_{i=1}^{\lceil K^{1/a} \rceil} s_i \le y \right) \mathbb{P}^{x_0}\left( F(T) \in dy \right) + \int_{\frac{\mu K^{1/a}}{\delta}}^\infty \mathbb{P}^{x_0}\left( \sum_{i=1}^{\lceil K^{1/a} \rceil} s_i \le y \right) \mathbb{P}^{x_0}\left( F(T) \in dy \right) \\
		=:& \mathcal{I}_1(K) + \mathcal{I}_2(K).
	\end{align*}
	
	For $ \mathcal{I}_1(K) $, we will use the fact that
	$$ 2 \exp\left(\frac{\pi^2}{4\delta m^2} \right) < \exp\left(\frac{\sqrt{2}\pi}{2m}\right) + \exp\left(-\frac{\sqrt{2}\pi}{2m}\right) $$
	for $ \delta$ large enough. Then, we obtain that $\mathcal{I}_1(K) $ decreases exponentially fast to zero. In fact, using Markov inequality
	\begin{align}
		I_1(K)\leq \mathbb{P}^{x_0}\left( \sum_{i=1}^{\lceil K^{1/a} \rceil} s_i \le \frac{\mu K^{1/a}}{\delta} \right)\leq
		e^{\frac{\mu K^{1/a}}{\delta}}
		\left(\mathbb{E}^{x_0}\left[e^{-s_1}\right]\right)^{\lceil K^{1/a} \rceil}\leq \left(\frac{2e^{\frac{\pi^2}{4\delta m^2}}}{e^{\frac{\sqrt{2}\pi}{2m}}+e^{-\frac{\sqrt{2}\pi}{2m}}}\right)^{\lceil K^{1/a} \rceil}.
	\end{align}
	Using \eqref{tail_bessel_integral}, we have:
	\begin{equation}
		I_2 (K)\le \mathbb{P}^{x_0}\left(F(T)\ge \frac{\mu K^{1/a}}{\delta} \right) \underset{K \rightarrow \infty}{\sim} C K^{-\frac{1}{2a}}, \ \ C = \frac{1}{\sqrt{2}\Gamma(1/2)} \left( \int_{\frac{r_0^2}{2T}}^\infty \frac{e^{-u}}{u}du \right) \left(\frac{\mu}{\delta }\right)^{-\frac{1}{2}} .
	\end{equation}
	%
	Since $0 < a <1/2$, the result follows.
	%
	
	Similarly, in the case $ a\geq 1/2 $, one has that 
	\begin{align*}
		\mathbb{P}^{x_0}(N^a \ge K) 
		\geq \int_{{2\mu K^{1/a}}}^\infty \mathbb{P}^{x_0}\left( \sum_{i=1}^{\lceil K^{1/a} \rceil} s_i \le y \right) \mathbb{P}^{x_0}\left( F(T) \in dy \right).
	\end{align*}
	For $ K $ large enough and $ y\geq  2\mu K^{1/a}$, we have $\mathbb{P}^{x_0}( \sum_{i=1}^{\lceil K^{1/a} \rceil} s_i \le y)\geq 1/2 $ by the Central Limit Theorem, and $ \mathbb{P}^{x_0}\left( F(T) \geq  {2\mu K^{1/a}} \right) \sim C K^{-1/(2a)}$ as $K \rightarrow \infty$ for some constant $ C>0 $ which implies the result in this case.
\end{proof}
From this result, we conclude that the average number of iterations of the algorithm in the reflected case is infinite.
A modification of the above algorithm with finite number of iterations in expectation is provided below.

\subsection{Proposition of modification of Algorithm II}
\label{algorithm_modified}

When the simulated radii $r_1, \ \ldots, \ r_n$ in Algorithm II of Section \ref{algorithm_for_reflected} become small, the process $(X_t)_{t\in [0,T]}$ goes from one boundary $\partial \mathcal{V}^{\pm}$ to the other boundary $\partial \mathcal{V}^{\mp}$ many times and the number of iterations  becomes high. 
In Section \ref{sec:higher_dim} is hinted that this problem is in fact specific to the dimension $2$.
To remedy this problem, we study the asymptotic behavior of $ \mathbb{P}^{x_0}(X_t \in dy) $ when $ r_0 $ is close to zero.  In fact, if the ratio $rr_0/t$ is small, then the density of the reflected Brownian motion expressed in \eqref{formula_reflected} can be approximated by the distribution obtained by taking only $n=0$:
\begin{equation}
\label{approximation_r_small}
\mathbb{P}^{x_0}(X_t \in dy) \simeq \frac{r}{t\alpha} e^{-\frac{r^2 + r_0^2}{2t}} I_0\left(\frac{rr_0}{t}\right)drd\theta.
\end{equation}
The function above can be renormalized so that it becomes a density which can be simulated using  inequality \eqref{majoration_I0} and the acceptance-rejection method with the reference density
\begin{equation}
\label{reference_density_modified}
\propto re^{-\frac{(r-r_0)^2}{2t}}drd\theta = \left((r-r_0)e^{-\frac{(r-r_0)^2}{2t}}dr + r_0e^{-\frac{(r-r_0)^2}{2t}}dr\right) d\theta.
\end{equation}
Then we modify Algorithm II as follows. We choose some small $\varepsilon \in (0,r_0)$ and:
\begin{itemize}
	\item Simulate the sequence of $(r_n)_n$ and $(\theta_n)_n$ as in Algorithm II in Section \ref{algorithm_for_reflected}.
	\item If after some iteration $n$, we have
	\begin{align}
		\label{eq:TTn}
	\frac{r_n^2}{T-T_n} < \varepsilon
	\end{align} (note that $T-T_n > 0$ is the remaining time of the simulated process after step $n$), then we directly simulate the final point $\bar X_T$ according to the approximation \eqref{approximation_r_small}:
	\begin{align}
	&  \mathbb{P}^{x_0}(\bar X_T \in dy | \ X_{T_n} = y_n)=C(\varepsilon) \frac{r}{(T-T_n)\alpha}e^{-\frac{r^2 +r_n^2}{2(T-T_n)}}I_0\left(\frac{rr_n}{T-T_n}\right) dr d\theta, \label{eq:6.6}\\
	& C(\varepsilon) :=\left(\int_0^\alpha \int_0^\infty \frac{r}{(T-T_n)\alpha}e^{-\frac{r^2 + r_n^2}{2(T-T_n)}}I_0\left(\frac{rr_n}{T-T_n}\right) dr d\theta \right)^{-1}.\nonumber
	\end{align}
Using that $I_0(x) \ge 1$ for all $x\geq 0$, one obtains that $C(\varepsilon) \le e^{r_n^2/(2(T-T_n))}$. This upper bound for $C(\varepsilon)$ is enough in order to implement the acceptance-rejection sampling method. In this case, the algorithm directly ends after one additional iteration. 
\end{itemize}

\begin{proposition}
\label{E_N_majoration_modified}
Denote by $\bar N$ the number of iterations of the modified algorithm, \eqref{eq:6.6}, then for any $ 1\leq a <p<2 $, there exists a quantity $C(a,p,x_0,T,m)>0$ such that for $\varepsilon$ small enough:
\begin{equation}
\label{majoration_EN}
\mathbb{E}^{x_0}[\bar N^a] \le \frac{C(a,p,x_0,T,m)}{\varepsilon^{p-1}}.
\end{equation}
\end{proposition}
\begin{proof}
	We use a similar argument as in the proof of Proposition \ref{prop:6}.
	Let $\zeta_\varepsilon := \inf \ \lbrace t \in [0,T]: \ R_t^2/(T-t) \le \varepsilon \rbrace$ and let $K \in \mathbb{N}$ fixed. Using Markov's inequality, we have:
	\begin{align*}
		\mathbb{P}^{x_0}(\bar{N}-1 \ge K^{1/a}) & \le \mathbb{P}^{x_0}\left(\sum_{i=1}^{\lceil K^{1/a} \rceil} s_i \le F({T \wedge \zeta_\varepsilon}), \ F({T \wedge \zeta_\varepsilon})\leq \frac{\mu K^{1/a}}{\delta}  \right)  +
		\mathbb{P}^{x_0}\left(F({T \wedge \zeta_\varepsilon})> \frac{\mu K^{1/a}}{\delta}  \right)\\
		& \le e^{\frac{\mu K^{1/a}}{\delta} }
		\left(\mathbb{E}^{x_0}\left[e^{-s_1}\right]\right)^{\lceil K^{1/a} \rceil}
		+\frac{\delta^p}{(\mu K^{1/a})^p}\mathbb{E}^{x_0}\left[F({T \wedge \zeta_\varepsilon})^p\right].
	\end{align*}
	From here the result follows by using Lemma \ref{majoration_bessel_expectation}.
\end{proof}

\begin{proposition}
\label{total_var_modified}
Denote by $\bar{X}_T$, the result obtained from the above approximation algorithm. Then, for $ \alpha\in (0,2 \pi) $ and $\varepsilon$ small enough, the error made by the approximation in total variation distance satisfies:
$$ d_{\textup{TV}}(\bar{X}_T, X_T) \le C_m \varepsilon^{\min\left(1,\frac{\pi}{2\alpha}\right)},$$
where $C_m$ is a constant which only depends on $m$.
\end{proposition}
\begin{proof}
	If the approximation is not used, then $\bar{X}_T = X_T$. Else, we denote by $n$ the step when the approximation is used, so that we simulate the last step starting from $(r_n\cos(\theta_n),r_n\sin(\theta_n))$ with remaining time $T-T_n =: t'$. Since the $\varepsilon$-condition is reached, we have $r_n^2/t' < \varepsilon$.
	Then if we denote by $ d_{\text{TV}_n}(\bar{X}_T, X_T) $ the total variation distance conditioned to the filtration $ \mathbb{F}_n $  up to time $ T_n $, and $d_n:=1/2$ for $n=0$ and $d_n:=1$ otherwise, we have:
	\begin{align*}
		d_{\text{TV}_n}(\bar{X}_T, X_T) & \le \int_{\mathcal{D}}  \Big| (C(\varepsilon)-1+1)\frac{r}{t'\alpha} e^{-\frac{r^2+r_n^2}{2t'}}I_0\left(\frac{rr_n}{t'}\right) \\
		& \quad \quad \quad - \frac{2r}{t'\alpha}e^{-\frac{r^2+r_n^2}{2t'}} \sum_{k=0}^\infty d_n I_{k\pi/\alpha}\left(\frac{rr_n}{t'}\right) \cos\left(\frac{k\pi\theta}{\alpha}\right) \cos\left(\frac{k\pi\theta_n}{\alpha}\right)\Big| dr d\theta \\
		 & \le \frac{|C(\varepsilon)-1|}{C(\varepsilon)} + \int_0^\infty \frac{2r}{t'}e^{-\frac{r^2+r_n^2}{2t'}} \sum_{k=1}^\infty I_{k\pi/\alpha}\left(\frac{rr_n}{t'}\right) dr .
	\end{align*}
	Using that for all $x \ge 0$, $I_0(x) \ge 1$, we have $C(\varepsilon) \le e^{r_n^2/(2t')}$. On the other hand, we have
	\begin{align*}
	C(\varepsilon)^{-1} \le \frac{1}{t'} \int_0^\infty r e^{-r^2/(2t')} I_0\left(\frac{rr_n}{t'}\right) dr = \int_0^\infty u e^{-u^2/2} I_0\left(\frac{r_n u}{\sqrt{t'}}\right) du. 
	\end{align*}
Using \eqref{eq:10} one immediately obtains that $ |I_0'(x)|+|I_0''(x)|\leq 2e^x $. Therefore the function $ A(\delta):=\int_0^\infty u e^{-u^2/2}I_0(\delta u)du $, $ \delta\geq 0 $ is twice differentiable. As
 $I_0(0)=1$ and $I'_0(0)=0$, we obtain $A(0)=1$, $A'(0)=0$ and then $A(\delta) = 1+O(\delta^2)$ as $\delta \to 0$ which gives $C(\varepsilon) \ge 1 - C\varepsilon$ for small enough $ \varepsilon>0 $.
	These bounds on $C(\varepsilon)$ imply that $|C(\varepsilon)-1|/(C(\varepsilon)\varepsilon)$ is bounded.
	For the second term, using the expression \eqref{eq:10} and a multiplicity property of the number of terms of the type $k\pi/\alpha$ in the interval $ [0,k] $, we have
	\begin{align*}
		& \sum_{k=1}^\infty I_{k\pi/\alpha}\left( \frac{rr_n}{t'}\right) = \left(\frac{rr_n}{2t'}\right)^{\frac{\pi}{\alpha}} \sum_{k=0}^\infty \sum_{m=0}^\infty \frac{1}{m! \Gamma\left(m+\frac{(k+1)\pi}{\alpha}+1\right)} \left(\frac{rr_n}{2t'}\right)^{2m+\frac{k\pi}{\alpha}}\\
		& \quad \le \left(\frac{rr_n}{2t'}\right)^{\frac{\pi}{\alpha}}\sum_{m=0}^\infty \frac{1}{m!}\left(\frac{rr_n}{2t'}\right)^{2m} \sum_{k=0}^\infty \frac{1}{\Gamma\left(\frac{(k+1)\pi}{\alpha}+1\right)}\left(\frac{rr_n}{2t'}\right)^{\frac{k\pi}{\alpha}} \le \left(\frac{rr_n}{2t'}\right)^{\frac{\pi}{\alpha}} e^{\left(\frac{rr_n}{2t'}\right)^2} \frac{\alpha}{\pi} e^{\frac{rr_n}{2t'}}.
	\end{align*}

	With this inequality and the change of variable $ u = rr_n/(2t')$, we obtain
	\begin{align*}
		d_{\text{TV}_n}(\bar{X}_T, X_T) \le&
		C\varepsilon+ \frac{8\alpha t'}{\pi r_n^2} e^{-\frac{r_n^2}{2t'}}\int_0^\infty u^{1+\frac{\pi}{\alpha}}e^{u^2+u}e^{-\frac{2t'}{r_n^2}u^2}du =: C\varepsilon + I(\varepsilon).
	\end{align*}

	The integral $ I(\varepsilon) $ converges if $\gamma \equiv \gamma(\varepsilon):= 2t'/r_n^2 > 1$, which is satisfied as soon as $\varepsilon < 2$. Then, denoting $\beta := 1 +\pi/\alpha$,
	\begin{align*}
		I(\varepsilon) \le& \frac{8\alpha t'}{\pi r_n^2} \int_0^\infty u^\beta e^{-(\gamma-1)u^2 + u}du = \frac{8\alpha t'}{\pi r_n^2}e^{\frac{1}{4(\gamma-1)}} \int_0^\infty u^\beta e^{-(\gamma-1)\left(u-\frac{1}{2(\gamma-1)}\right)^2}du\\
		\le & \frac{8\alpha t'}{\pi r_n^2}e^{\frac{1}{4(\gamma-1)}} \left( \int_{\frac{1}{2(\gamma-1)}}^\infty u^\beta e^{-(\gamma-1)u^2} du + \left(\frac{1}{2(\gamma-1)}\right)^\beta \int_{-\frac{1}{2(\gamma-1)}}^{0} e^{-(\gamma-1)u^2}du \right) \\
		\le & \frac{4\alpha \gamma}{\pi} e^{\frac{1}{4(\gamma-1)}} \left( \frac{1}{\sqrt{\gamma-1}^{1+\beta}} \int_0^\infty v^\beta e^{-v^2}dv + \frac{1}{2^\beta(\gamma-1)^{\beta+1/2}}  \int_{-\frac{1}{2\sqrt{\gamma-1}}}^{0} e^{-v^2} dv \right) .
	\end{align*}
	The first term converges to zero as  $O(\varepsilon^{\pi/(2\alpha)})$   while the second term converges to zero as  $O(\varepsilon^{1+ \pi/\alpha})$ when $ \varepsilon\to 0 $. Taking all the elements into account we obtain that 
	$$ d_{\text{TV}}(\bar{X}_T, X_T) \underset{\varepsilon \rightarrow 0}{=} O\left(\varepsilon^{\min\left(1,\frac{\pi}{2\alpha}\right)}\right) .$$
\end{proof}

\section{Adaptation of the algorithms to general processes}
\label{section:extensions}

\subsection{Stopped Brownian motion with constant drift}
\label{girsanov_stopped}

Let $b \in \mathbb{R}^2$ and consider the Brownian motion with drift: $\widetilde{W}_t := W_t + bt$. Then by Girsanov's theorem, $\widetilde{W}$ is a Brownian motion under the probability
\begin{equation}
\label{eq:girsanov_proba_def}
\mathbb{Q}^{x_0}_{|\mathcal{F}_t} := e^{-b \cdot W_t - b^2t/2} \mathbb{P}^{x_0}_{|\mathcal{F}_t},
\end{equation}
so that for any bounded test function $f : \mathcal{D}\times [0,T] \rightarrow \mathbb{R}$:
$$ \mathbb{E}^{x_0} \left[e^{-b \cdot \widetilde{W}_{T\wedge \widetilde\tau} + \frac{b^2T\wedge \widetilde\tau }{2}} f(\widetilde{W}_{T \wedge \widetilde\tau}, T\wedge \widetilde{\tau}) \right] = \mathbb{E}^{x_0} \left[f(W_{T \wedge \tau}, T\wedge \tau)\right] .$$
So we have:
$$ \mathbb{E}^{x_0} \left[f(\widetilde{W}_{T\wedge \widetilde{\tau}},T\wedge \widetilde{\tau})\right] =
\mathbb{E}^{x_0} \left[ e^{b \cdot W_{T\wedge \tau}-\frac{b^2T\wedge \tau}{2}} f(W_{T\wedge \tau}, T\wedge \tau) \right] .$$
To simulate the Brownian motion with drift stopped in the wedge $\mathcal{D}$, we proceed as follows:
\begin{enumerate}
	\item Simulate $W_{T\wedge\tau}$ and $T \wedge \tau$ for the stopped process without drift according to Algorithm I in Section \ref{algorithm_for_stopped}.
	\item The result of the simulation becomes
\begin{equation}	
	e^{-\frac{b^2 T \wedge \tau}{2}} e^{b \cdot W_{T\wedge\tau}} f(W_{T\wedge\tau})
\end{equation}
\end{enumerate}

We can also deduce an explicit formula for the density of the stopped Brownian motion with drift using \eqref{formula_stopped}. In fact, for $y \in \mathcal{D}$, we have
\begin{equation}
\mathbb{P}^{x_0}(\widetilde W_t  \in dy, \ \widetilde\tau > t) = \frac{2r}{t \alpha}e^{-\frac{b^2t}{2}} e^{b \cdot y} e^{-(r^2 + r_0^2)/2t} \sum_{n=1}^{\infty} I_{n\pi / \alpha}\left(\frac{rr_0}{t}\right) \sin\left(\frac{n\pi \theta}{\alpha}\right) \sin\left(\frac{n \pi \theta_0}{\alpha}\right) dr d\theta ,
\end{equation}
with the same notations as before: $ \widetilde\tau:=\inf\{s: \ W_s + bs\in\partial {\mathcal D}\} $, $x = (r_0 \cos(\theta_0), r_0\sin(\theta_0))$ and $y = (r\cos(\theta),r\sin(\theta))$.

\subsection{Adaptation of the simulation algorithms for It\={o} processes}
\label{sec:ito_process}

In this section, we present some extensions of the above simulation methods applied to the approximation of stopped and reflected diffusions at some fixed time $ T>0 $. Let us start, first with the case of killed diffusion. 

Consider the following It\={o} process in $\mathbb{R}^2$:
\begin{equation}
\left\{ \begin{array}{l}
dY_t = b(Y_t)dt + \sigma(Y_t)dW_t \\
Y_0 = x_0,
\end{array} \right.
\end{equation}
where $(W_t)_{t\geq 0}$ is a standard two-dimensional Brownian motion and $x_0 \in \mathcal{D}$. We define $\tau$ to be the exit time of $(Y_t)_{t\geq 0}$ from $\mathcal{D}$.

\medbreak

\textbf{Simulation of $Y_{T \wedge \tau}$:} Choose $ t_i:=Ti/n$, $ i=0,...,n $,  a uniform partition of $[0,T]$ and consider the following Euler-Maruyama scheme. Given $\bar{Y}_{t_k}\in \mathrm{int}(\mathcal{D})$ (with $\bar{Y}_{t_0} = x_0$) and supposing that $ \bar{\tau}:=\inf\{s;\bar{Y}_s\in \partial\mathcal{D}\} >t_k$, simulate $\bar{Y}_{t_{k+1}}$ as the final value of the following Brownian motion with drift, stopped on the wedge $\mathcal{D}$, $ t \in [t_k,t_{k+1}] $:
\begin{equation}
\label{euler_scheme}
\bar{Y}_{t\wedge \bar\tau}:= \bar{Y}_{t_k} + b(\bar{Y}_{t_k})(t-t_k) + \sigma(\bar{Y}_{t_k})\cdot (W_{t\wedge \bar\tau}-W_{t_k}),
\end{equation}
 This is done using the algorithm in Section \ref{girsanov_stopped}. More precisely, we simulate $\bar{Y}_{T \wedge \bar{\tau}}$ taking $b \equiv 0$ while keeping track of the Brownian increments; then instead of $f(\bar{Y}_{T \wedge \bar{\tau}})$, the result of the simulation becomes
\begin{equation}
\exp\left( - \frac{1}{2} \int_0^{T \wedge \bar{\tau}} |\sigma^{-1}(\bar{Y}_s) \cdot b(\bar{Y}_s)|^2 ds + \int_0^{T \wedge \bar{\tau}} (\sigma^{-1}(\bar{Y}_s) \cdot b(\bar{Y}_s)) \cdot dW_s \right)  f(\bar{Y}_{T \wedge \bar{\tau}}).
\end{equation} 
 Note that we have to take into account the decorrelation step and the change of angle as described in Section \ref{sec:3.1} at every step. If the process \eqref{euler_scheme} exits the wedge $\mathcal{D}$ in the interval $ [t_k,t_{k+1}] $ then the algorithm directly stops. This scheme has weak order one. That is, for any $ f,b,\sigma \in \mathcal{C}^5_b (\bar{\mathcal{D}})$ one has that there exists a constant $ C_f>0 $
 \begin{align*}
 	\left|\mathbb{E}^{x_0}\left[f(Y_{T\wedge\tau})\right]-\mathbb{E}^{x_0}\left[f(\bar{Y}_{T\wedge\bar{\tau}})\right]\right|\leq C_fn^{-1}.
 \end{align*}
The proof is straightforward if one follows the same line of proof as in \cite{Gobet}. In particular, see Section 2.2.1 and note that our case is simpler as the proposed scheme does not have the possibility of touching the boundary before it is stopped or reaches $ T $.  For a discussion about the associated partial differential equation with Dirichlet conditions, see the discussion right after \eqref{reflected_pde}.

\medbreak

Since the Girsanov change of measure uses an exponential function, it may not be suitable to processes with large drifts as it may increase the variance. For this reason, we also propose a two-step Euler-Maruyama scheme for the stopped process.

\textbf{Simulation of $Y_{T \wedge \tau}$:} Choose $ t_i:=Ti/n$, $ i=0,...,n $,  a uniform partition of $[0,T]$ and consider the following two-step Euler-Maruyama scheme. Given $\bar{Y}_{t_k}\in \mathrm{int}(\mathcal{D})$ (with $\bar{Y}_{t_0} = x_0$) and supposing that $ \bar{\tau}:=\inf\{s;\bar{Y}_s\in \partial\mathcal{D}\} >t_k$ and that $\tilde{\tau}:=\inf\{s;\tilde{Y}_s\in \partial\mathcal{D}\} >t_k$, simulate $\bar{Y}_{t_{k+1} \wedge \bar{\tau}}$ in two steps as follows for $ t \in [t_k,t_{k+1}] $:
\begin{align}
	\tilde{Y}_{t \wedge \bar{\tau}} & := \bar{Y}_{t_k} + b(\bar{Y}_{t_k}) (t \wedge \bar{\tau} - t_k) \\
	\bar{Y}_{t \wedge \tilde{\tau}} & := \tilde{Y}_{t_{k+1} \wedge \bar{\tau}} + (W_{t \wedge \bar{\tau}}-W_{t_k}).
\end{align}
The second step is performed using the algorithm in Section \ref{algorithm_for_stopped}. Note that we have to take into account the decorrelation step and the change of angle as described in Section \ref{sec:3.1}. If  the process exits the wedge $\mathcal{D}$ in the interval $ [t_k,t_{k+1}] $ then the algorithm directly stops. This scheme has weak order one. That is, for any $ f,b,\sigma \in \mathcal{C}^5_b (\bar{\mathcal{D}})$, there exists a constant $ C_f>0 $ such that
 \begin{align}
 	\left|\mathbb{E}^{x_0}\left[f(Y_{T\wedge\tau})\right]-\mathbb{E}^{x_0}\left[f(\bar{Y}_{T\wedge\bar{\tau}})\right]\right|\leq C_fn^{-1}.
 \end{align}

The proof is similar to the proof for the convergence of the two-step Euler scheme for the reflected processes in Proposition \ref{prop:two_step_reflected}.

\medbreak

Now let $(Z_t)_{t\geq 0}$ to be the solution of a  stochastic differential equation with drift coefficient $ b $ which is normally reflected on the boundary of the wedge $\mathcal{D}$. That is, 
\begin{align*}
dZ_t = b(Z_t)dt + dW_t+dL_t.
\end{align*}
For details on existence and uniqueness for this equation we refer to \cite{Saisho}.
\medbreak
In this case, we do not know how to extend the simulation method using the Girsanov change of measure; instead we propose a two-step scheme.

\textbf{Simulation of $Z_T$:} Choose $ t_i:=Ti/n$, $ i=0,...,n $ a uniform partition of $[0,T]$ and consider the following two-step scheme. Given $\bar{Z}_{t_k}$ (with $\bar{Z}_{t_0} = x_0$), simulate $\bar{Z}_{t_{k+1}}$ as the  final value of the reflection on the wedge $\mathcal{D}$ of the Brownian motion and drift in two steps as follows for $ t \in [0,t_{k+1}-t_k] $
\begin{align}
	\label{eq:two_step:1}
	\tilde{Z}_{t+t_k} & :=\bar{Z}_{t_k} + b(\bar{Z}_{t_k})(t-t_k)+L^1_t-L^1_{t_k}\\
	\label{eq:two_step:2}
	\bar{Z}_{t+t_k}& := \tilde{Z}_{t_{k+1}} +  (W_t-W_{t_k})+L^2_t-L^2_{t_k}.
\end{align}
The terms $ L^i $, $ i=1,2 ,$ denote the respective local time terms for each of the two steps. 
 The simulation of the second step of this algorithm uses the argument described in either  Sections \ref{algorithm_for_reflected} or \ref{algorithm_modified}. The weak rate of convergence for these methods is as follows:
 \begin{proposition}
 \label{prop:two_step_reflected}
 Assume that $ f,b\in \mathcal{C}^5_b (\bar{\mathcal{D}})$ then 
 	\begin{align}
 		\left|\mathbb{E}^{x_0}\left[f({Z}_{T})\right]-\mathbb{E}^{x_0}\left[f(\bar{Z}_{T})\right]\right|\leq C_fn^{-1}.
 		\label{eq:712}
 	\end{align}
 Denoting by $ \hat {Z}_T$, the result of the algorithm using the approximation described in Section \ref{algorithm_modified}, we have
 	\begin{align}
 		\left|\mathbb{E}^{x_0}\left[f({Z}_{T})\right]-\mathbb{E}^{x_0}\left[f(\hat{Z}_{T})\right]\right|\leq C_{f}\left(n^{-1}+\varepsilon^{\min\left(1,\frac{\pi}{2\alpha}\right)}\right).
 		\label{eq:713}
 	\end{align}
 \end{proposition}
\begin{proof}
	The proof of \eqref{eq:713} is a small modification of \eqref{eq:712}. Therefore, we prove the rate of convergence for the  algorithm provided in Section \ref{algorithm_for_reflected}.
	
	 The argument follows as in the classical diffusion case which can be found in the proof of the main result in \cite{talay1990}. We use this argument and refer the reader to \cite{talay1990} for more details. First, consider $ u $ to be the solution of the following backward partial differential equation with Neumann conditions:
\begin{align}
	\label{pde:2}
		\begin{array}{ll}
		\partial_t u (t,x)  + \mathcal{L}u(t,x)=0, \ & (t,x)\in [0,T] \times \mathcal{D}  \\
		u(T,x) = f(x), \ & x\in \mathcal{D}\\
		\nabla u(t,x) \cdot n(x) = 0, \ & t > 0, \ x \in \partial \mathcal{D} .
	\end{array}
\end{align}
Here, $ \mathcal{L}u(t,x):= b(x)\nabla u(t,x)+\frac{1}{2} \Delta u (t,x) $.
Following the same discussion as in \eqref{reflected_pde}, one obtains that $ u \in \mathcal{C}^{7/2,7}([0,T]\times \bar{\mathcal{D}})$. This property will be used when doing Taylor's expansions for $ u $. 
Next, note that 
\begin{align}
	\label{eq:sum}
	\mathbb{E}^{x_0}\left[f({Z}_{T})\right]-\mathbb{E}^{x_0}\left[f(\bar{Z}_{T})\right]=&\sum_{i=0}^{n-1}\left(\mathbb{E}^{x_0}\left[u(t_i,\bar{Z}_{t_i})-u(t_{i+1},\bar{Z}_{t_{i+1}})\right]\right)=:-\sum_{i=0}^{n-1}A_{i}.
\end{align}
We will now consider each term within the above sum. Using It\=o-Tanaka formula (here one uses \eqref{pde:2} to cancel all local time terms), we have
\begin{align*}
	A_i & = \mathbb{E}^{x_0}\left[u(t_{i},\tilde{Z}_{t_{i+1}})-u(t_i,\bar{Z}_{t_i})+\int_{t_i}^{t_{i+1}}\left(\partial_t +\frac 12\Delta  \right)u(s,\bar{Z}_s)ds\right]\\
	& = \int_{t_i}^{t_{i+1}}\mathbb{E}^{x_0}\left[b(\bar{Z}_{t_i})\nabla u(t_i,\tilde{Z}_s)+\left(\partial_t+\frac 12\Delta  \right)u(s,\bar{Z}_s)\right]ds.
\end{align*}
The following step relies on using the Taylor expansion for each of the terms involving $ \nabla u(t_i,\tilde{Z}_s) $ and $ \left(\partial_t +(1/2)\Delta  \right)u(s,\bar{Z}_s) $ at $ (t_i,\bar{Z}_{t_i}) $ together with the fact that $ u $ solves \eqref{pde:2}. Without going into much detail, let us consider the expansion of the term $  \partial_tu(s,\bar{Z}_s)$:
\begin{align*}
	\partial_tu(s,\bar{Z}_s) & = \partial_tu(t_i,\bar{Z}_{t_i}) + \int_0^1 \nabla\partial_tu(t_i,\alpha\bar{Z}_s+(1-\alpha)\bar{Z}_{t_i})\cdot (\bar{Z}_s-\bar{Z}_{t_i})d\alpha \\
	& \quad + \int_0^1 \partial_t\partial_t u(\alpha s+(1-\alpha) t_i, \bar{Z_s})\cdot (s-t_i)d\alpha
\end{align*}
Each of the above derivatives of $ u $ is bounded. In the case of the increments of $ \bar{Z}_s-\bar{Z}_{t_i}=(W_s-W_{t_i})+L^2_s-L^2_{t_i}+\tilde{Z}_{t_{i+1}}-\bar{Z}_{t_i} $ one has to take one further step in the Taylor expansion to be able to use the fact that the expectation of the Brownian increment is zero. Namely,
\begin{align*}
& \int_0^1 \nabla\partial_tu(t_i,\alpha\bar{Z}_s+(1-\alpha)\bar{Z}_{t_i})\cdot (\bar{Z}_s-\bar{Z}_{t_i})d\alpha \\
& \quad = \nabla \partial_t u(t_i, \bar{Z}_{t_i}) \cdot (\bar{Z}_s-\bar{Z}_{t_i}) + \int_0^1 (1-\alpha) (\bar{Z}_s-\bar{Z}_{t_i})^\top \cdot \nabla^2 \partial_t u(t_i,\alpha\bar{Z}_s+(1-\alpha)\bar{Z}_{t_i}) \cdot (\bar{Z}_s-\bar{Z}_{t_i}) d\alpha.
\end{align*}

After this and canceling the first terms of the expansion using \eqref{pde:2} and taking expectations we see that the sum of most terms is of order $ O(n^{-2}) $. The remaining terms  are bounded by  $  C\int_{t_i}^{t_{i+1}}(L^j_s-L^j_{t_i})ds$ for $ j=1,2 $. Considering these terms within \eqref{eq:sum}, one obtains for $ L^j_T = \sum_{i=0}^{n-1}(L^j_{t_{i+1}}-L^j_{t_i})$ that these sums are bounded by 
\begin{align*}
\sum_{i=0}^{n-1}\int_{t_i}^{t_{i+1}}\mathbb{E}^{x_0}[L^j_s-L^j_{t_i}]ds\leq n^{-1}\mathbb{E}^{x_0}[L^j_T]\leq Cn^{-1},
\end{align*}
where we used that $t \mapsto L^j_t$ is increasing.
The last inequality follows from an estimate for the Skorohod problem in convex domains (see \cite[Theorem 4.2]{Saisho}). In the case of the algorithm described in Section \ref{algorithm_modified} one may modify the above arguments adding at each time $ t_i $, the possibility of ending the simulation in the partition interval using the criteria provided in \eqref{eq:TTn}. Using Proposition \ref{total_var_modified} one obtains the result.
\end{proof}

\section{Simulations}
\label{section:simulations}

We implement the algorithms\footnote{The programs with a demonstration notebook are available at \url{https://github.com/Bras-P/simulation_reflected_brownian_motion_wedge}} described in Sections \ref{algorithm_for_stopped} and \ref{algorithm_for_reflected} as well as the approximated version in Section \ref{algorithm_modified}. We also give simulation results obtained with Metzler's algorithm \cite[Section 2.1]{metzler2008}. This last algorithm simulates $W_\tau$ (instead of $W_{\tau \wedge T}$) directly for a wedge of general angle $\alpha$, by:
\begin{enumerate}
	\item Simulate the radius $|W_\tau|$ using \eqref{eq:metzler_radius}.
	\item Conditionally simulate $\tau$ by acceptance-rejection, approximating the infinite sum by a partial sum, and using a Cauchy distribution as reference density.
\end{enumerate}
However, Metzler's algorithm cannot be adapted to simulate $W_{\tau \wedge T}$ and is biased for the simulation of $\tau$ (see \cite[Proposition 2.1.8]{metzler2008}).

In Tables \ref{table:sim:1} and \ref{table:sim:2}, we consider the Monte Carlo estimation of various expectations. We choose positive test functions in order to avoid cancellations and $r_0$ so that $\mathbb{E}^{x_0}[\tau]$ is of the same order as $T$, so that the stopped and reflected processes are significantly different than the standard Brownian motion.

\begin{table}
\centering
\begin{tabular}{cccccc}
\hline 
& $\mathbb{E}^{x_0}$ & 95 \% interval & Time (s) & MC iterations & $\mathbb{E}^{x_0}[N]$ \\ 
\hline 
$\mathbb{E}^{x_0}[f(W_\tau)]$ with \cite{metzler2008} & 3.473 & $\pm$ 0.082 & 0.67 & 50000 & 1 \\ 
\hline 
$\mathbb{E}^{x_0}[f(W_\tau)]$ & 3.489 & $\pm$ 0.085 & 1.04 & 50000 & 1.45 \\ 
\hline
$\mathbb{E}^{x_0}[\tau]$ with \cite{metzler2008} & 0.742 & $\pm$ 0.030 & 7.06 & 20000 & - \\
\hline
$\mathbb{E}^{x_0}[\tau]$ & 0.590 & $\pm$ 0.024 & 4.92 & 20000 & - \\
\hline
$\mathbb{E}^{x_0}[f(W_{\tau \wedge T})]$ & 2.980 & $\pm$ 0.049 & 5.79 & 10000 & 1.37 \\
\hline
$\mathbb{E}^{x_0}[(W_{\tau \wedge T}) \cdot e_1]$ & 1.441 & $\pm$ 0.012 & 6.26 & 10000 & 1.37 \\
\hline
$\mathbb{E}^{x_0}[f(X_T)]$, $\varepsilon=0$ & 4.138 & $\pm$ 0.224 & 13.2 & 1000 & 75.6 \\
\hline
$\mathbb{E}^{x_0}[f(X_T)]$, $\varepsilon=0.03$ & 4.313 & $\pm$ 0.072 & 28.2 & 10000 & 5.11 \\
\hline
\end{tabular}
\caption{Simulations with $\alpha = 0.9$, $r_0=1.5$, $\theta_0=0.3$, $f(x,y)=x^2+y^2$, $T=1$.}
\label{table:sim:1}
\end{table}

\begin{table}
\centering
\begin{tabular}{cccccc}
\hline
& $\mathbb{E}^{x_0}$ & 95 \% interval & Time (s) & MC iterations & $\mathbb{E}^{x_0}[N]$ \\ 
\hline 
$\mathbb{E}^{x_0}[f(W_{\tau \wedge T})]$ & 0.195 & $\pm$ 0.003 & 8.85 & 10000 & 1.28 \\
\hline
$\mathbb{E}^{x_0}[f(X_T)]$, $\varepsilon=0.03$ & 0.117 & $\pm$ 0.003 & 9.26 & 5000 & 2.73 \\
\hline
\end{tabular}
\caption{Simulations with $\alpha = 0.58$, $r_0=3$, $\theta_0=0.4$, $f(r\cos(\theta),r\sin(\theta))=\sin^2(\theta)$, $T=1$.}
\label{table:sim:2}
\end{table}

We note that for the estimation of $\mathbb{E}[f(W_\tau)]$, our simulation method finds a similar value as the algorithm proposed by Metzler, as both are exact simulation methods. However, for the estimation of $\mathbb{E}[\tau]$, the bias in the method proposed by Metzler seems to be significant.

As a second way to check our algorithm, we also simulated the projection on the first axis of $W_{\tau \wedge T}$. Note that the simulation gives a result close to the initial point $1.5 \cdot \cos(0.3) \simeq 1.4330 $ which is reported on the sixth line of Table \ref{table:sim:1}. This is because  $t \mapsto W_{\tau \wedge t}$ is a martingale.

For the estimation of $\mathbb{E}[f(X_T)]$ (reflected Brownian motion), the exact algorithm takes too much time, as hinted in Proposition \ref{prop:6}, hence the need to use the approximation version from Section \ref{algorithm_modified}. We only give a simulation example with 1000 samples, as we could not get a more proper estimation. Indeed, if we increase the number of samples, there appear trajectories where $N$ becomes large and where the algorithm does not stop in reasonable time.
\begin{figure}[!ht]
	{%
		\includegraphics[width=0.49\textwidth]{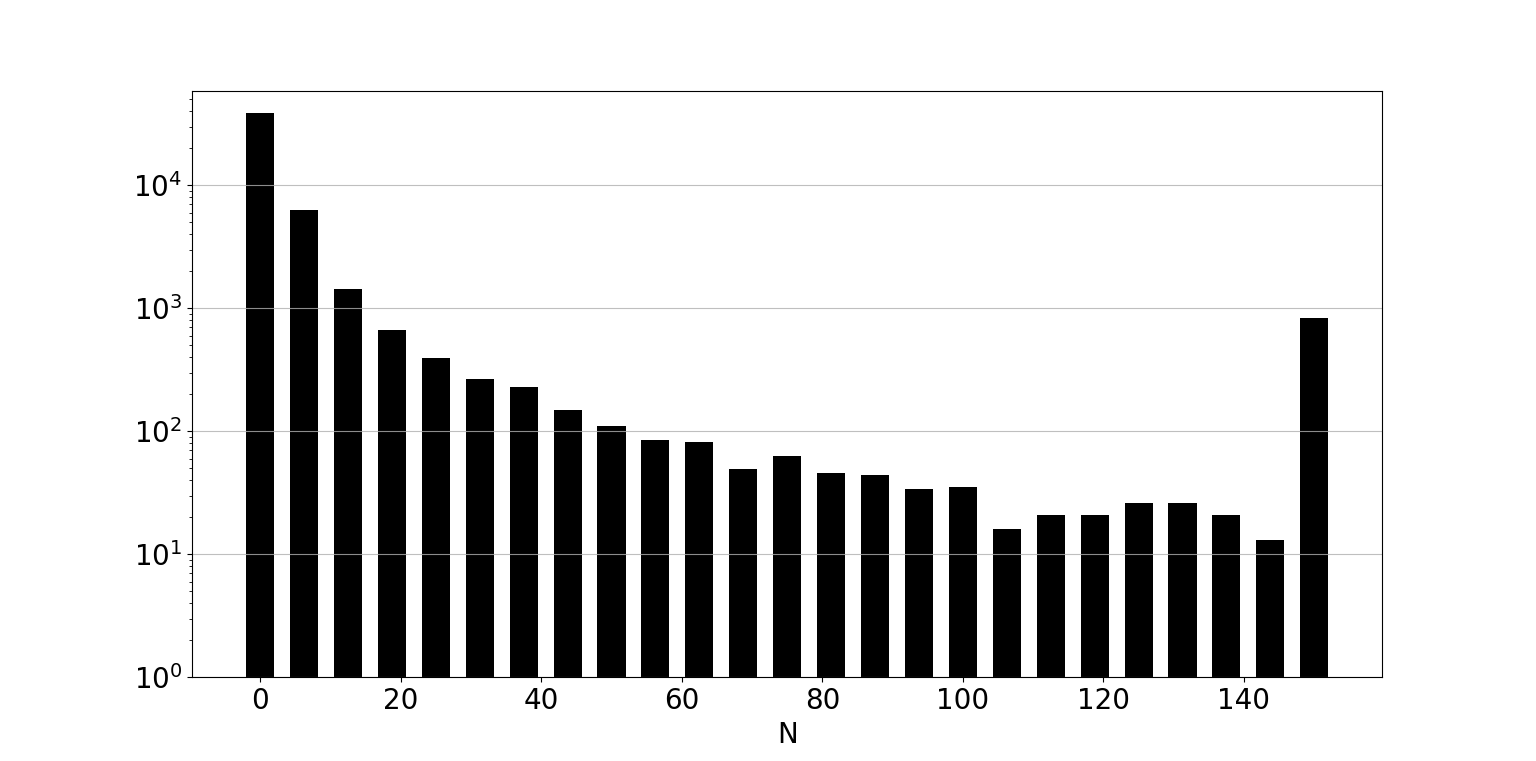}
	}
	\hfill\centering
	{%
		\includegraphics[width=0.49\textwidth]{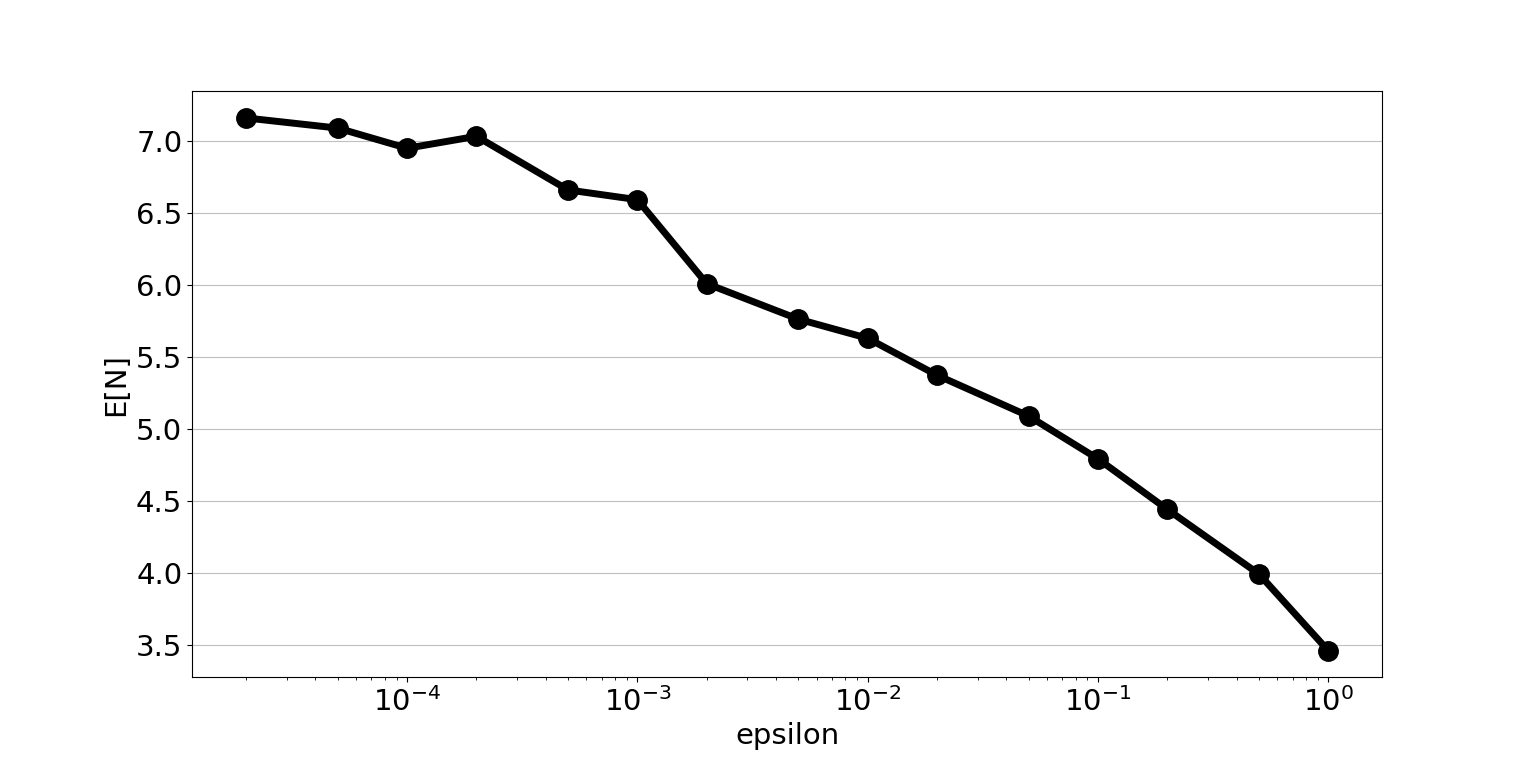}
	}
	\caption{(left) Histogram of $N$ for the exact reflected algorithm with 50000 Monte Carlo iterations. The horizontal axis represents the values of $N$ while the vertical axis gives the number of trajectories such that $N$ is in the interval noted in the horizontal axis.  The last bar counts all the values greater than $150$. On the right, average number of iterations in function of $\varepsilon$. The parameters are the same as in Table \ref{table:sim:1} and for each value of $ \varepsilon $, 50000 Monte Carlo iterations are used. }
	\label{fig:sim_N_epsilon}
\end{figure}

To have a better idea of the behavior of the algorithm, we provide on the left side of  Figure \ref{fig:sim_N_epsilon}, the histogram of $N$ for the exact Algorithm II from Section \ref{algorithm_for_reflected}; if, for a trajectory, $N$ exceeds 150, then we stop the algorithm. We observe that for most of the trajectories, $N$ is not too large (does not exceed 5). However, for a few iterations (around 0.2\%) which corresponds to the case where the radius $r$ becomes small, $N$ is large, which slows down the algorithm.
On the right side of Figure \ref{fig:sim_N_epsilon} we trace $\mathbb{E}^{x_0}[N]$ for the approximated reflected algorithm in function of $\varepsilon$. However, we could not trace the bias in function of $\varepsilon$, as it seems negligeable in comparison with the size of the confidence interval. This is why we do not need to lengthen the simulation time of each trajectory by choosing a very small $\varepsilon$; $\varepsilon=0.02$ or $\varepsilon=0.05$ with $\mathbb{E}^{x_0}[N] \approx 5$ is sufficient.

In Table \ref{table:sim_ito_process}, we implement the algorithms adapted to general It\={o} processes from Section \ref{sec:ito_process}.

\begin{table}
\centering
\begin{tabular}{cccccc}
\hline
& $\mathbb{E}^{x_0}$ & 95 \% interval & Time (s) & MC iterations & Discretization step \\ 
\hline 
$\mathbb{E}^{x_0}[f(Y_{\tau \wedge T})]$ [Girsanov] & 2.83 & $\pm$ 0.06 & 456 & 5000 & T/500 \\
\hline
$\mathbb{E}^{x_0}[f(Y_{\tau \wedge T})]$ [Two-step] & 2.84 & $\pm$ 0.06 & 457 & 5000 & T/500 \\
\hline
$\mathbb{E}^{x_0}[f(Z_T)]$, $\varepsilon=0.01$ & 3.77 & $\pm$ 0.09 & 1140 & 5000 & T/500 \\
\hline
\end{tabular}
\caption{Simulation of the process $dY_t = -\mu(Y_t-\kappa) + \sigma dW_t$, with $\alpha = 0.9$, $r_0=1.5$, $\theta_0=0.3$, $\mu=(0.1,0.2)$, $\kappa=(0.7,0.5)$, $\sigma=I_2$, $f(x,y)=x^2+y^2$, $T=1$. $ Z $ corresponds to the reflected process in the same domain. The first stopped process is simulated using the Girsanov method while the second one using the two-step Euler-Maruyama scheme.}
\label{table:sim_ito_process}
\end{table}

%
%
%

\appendix

\section{Appendix: Auxiliary Lemmas }
\label{sec:A}

We start with a simple lemma about the simulation of a r.v. $ X $ using only partial information of a previously simulated r.v. $ Y .$

\begin{proposition}
	\label{prop:forget_step}
	Let $X$ and $Y$ be two random variables. We want to simulate $X$. Let $A$ be a deterministic set. We consider the following algorithm :
	\begin{enumerate}
		\item Simulate Y.
		\item If $Y \in A$, simulate $X$ conditionally to the value of $Y$ in the previous step, and return $X$.
		\item If $Y \notin A$, simulate $X$ conditionally to the event $Y \notin A$, and return $X$.
	\end{enumerate}
	Let us denote $\widetilde{X}$ the value returned by the algorithm. Then:
	$$ \widetilde{X} \overset{\mathscr{L}}{=} X .$$
\end{proposition}
\begin{proof}
	The random variable $\widetilde{X}$ is defined conditionally to $Y$ as:
	$$ \mathbb{P}(\widetilde{X} \in dx | \ Y=y ) = \left\lbrace \begin{array}{ll}
		\mathbb{P}(X \in dx | \ Y=y) & \text{ if } y \in A \\
		\mathbb{P}(X \in dx | \ Y \notin A) & \text{ if } y \notin A
	\end{array} \right. $$
	Then:
	\begin{align*}
		\mathbb{P}(\widetilde{X} \in dx) & = \int \mathbb{P}(\widetilde{X} \in dx | \ Y=y) \mathbb{P}(Y \in dy) \\
		& = \int_{y \in A} \mathbb{P}(X \in dx | \ Y=y) \mathbb{P}(Y \in dy) + \int_{y \notin A} \mathbb{P}(X \in dx | \ Y \notin A) \mathbb{P}(Y \in dy) \\
		& = \mathbb{P}(X \in dx , \ Y \in A) + \mathbb{P}(X \in dx| \ Y \notin A) \mathbb{P}(Y \notin A) = \mathbb{P}(X \in dx) .
	\end{align*}
\end{proof}
\subsection{Estimates on Bessel processes}

\begin{proposition}
	For all $x \ge 0$, we have:
	\begin{align}
		\label{majoration_I0}
	1\leq 	I_0(x) \le e^x .
	\end{align}
	Furthermore, for  $x \ge 0$ and $\nu \ge 0$, we have:
	\begin{align}
		\label{majoration_In} I_\nu(x) \le e^x + \frac{1}{\pi(\nu+x)} .
	\end{align}
\end{proposition}
\begin{proof}
	The inequality follows from the integral representation \cite[p.181]{watson_1944}:
	$$ I_\nu(x) = \frac{1}{\pi} \int_0^\pi e^{x \cos\theta} \cos(\nu \theta) d\theta - \frac{\sin(\nu \pi)}{\pi} \int_0^\infty e^{-x \cosh(t) - \nu t} dt .$$
\end{proof}


Now we provide two results on estimates related to the law of $F(T)=\int_0^T \frac{ds}{R_s^2}$.

\begin{proposition}
	\label{prop:35}
	Let $(R_t)_{t\geq 0}$ be a Bessel process of dimension $2$, or equivalently, of index $0$. Then 
	\begin{equation}
		\label{tail_bessel_integral}
		\mathbb{P}^{r_0}\left[ F(T)\ge x \right] \underset{x \rightarrow \infty}{\sim} \frac{1}{\sqrt{2}\Gamma(1/2)} \left(\int_{\frac{r_0^2}{2T}}^\infty \frac{e^{-u}}{u} du \right) x^{-1/2} .
	\end{equation}
\end{proposition}

\begin{proof}
	{\it Step 1: }We use \cite{jeanblanc2010}, Proposition 6.2.5.1, with $a=0$ and $\nu = 0$: 
	$$ 
	\mathbb{E}^{r_0}\left[ e^{-\frac{\lambda^2}{2} F(T)} \right] = \frac{r_0^\lambda}{\Gamma(\lambda/2)} \int_0^\infty v^{\frac{\lambda}{2}-1}(1+2vT)^{-(1+\lambda)}e^{-\frac{r_0^2v}{1+2vT}} dv.$$
	Performing the change of variable $w := \frac{Tv}{1+2vT}$, so that $v = \frac{w}{T(1-2w)}$, we have for all $\lambda >0$:
	\begin{align}
		\label{bessel_laplace_transform}
		\mathbb{E}^{r_0}\left[ e^{-\frac{\lambda^2}{2}F(T) } \right] = \frac{r_0^\lambda }{\Gamma(\lambda/2)T^{\lambda/2}} \int_0^{1/2} w^{\frac{\lambda}{2}-1}(1-2w)^{\lambda/2}e^{-r_0^2 w/T} dw.
	\end{align}
	
	{\it Step 2: }In fact, we prove the asymptotic expansion:
	\begin{align}
		\label{eq:exp1}
		\mathbb{E}^{r_0}\left[ e^{-\frac{\lambda^2}{2}F(T) } \right] \underset{\lambda \rightarrow 0^+}{=} 1 - \frac{\lambda}{2} \int_{\frac{r_0^2}{2T}}^\infty \frac{e^{-u}}{u} du \ + O(\lambda^2).
	\end{align}
	We have for all $x>0$, $\Gamma(x+1) = x\Gamma(x)$ and $\Gamma(1)=1$ so that
	$$ \Gamma(x) \underset{x \rightarrow 0^+}{\sim} \frac{1}{x}.$$
	Then, using \eqref{bessel_laplace_transform} and the change of variable $w = Tu/r_0^2$:
	$$ \mathbb{E}^{r_0}\left[ e^{-\frac{\lambda^2}{2}F(T)} \right] = \frac{1}{\Gamma\left(\frac{\lambda}{2}\right)} \int_0^{\frac{r_0^2}{2T}} u^{\frac{\lambda}{2}-1}\left(1-\frac{2Tu}{r_0^2}\right)^{\lambda/2} e^{-u} du .$$
	Using the definition of the Gamma function and decomposing the integral in two parts, we have
	$$1 - \mathbb{E}^{r_0}\left[ e^{-\frac{\lambda^2}{2} F(T)} \right] = \frac{1}{\Gamma\left(\frac{\lambda}{2}\right)} \left[ \int_0^{\frac{r_0^2}{2T}}u^{\frac{\lambda}{2}-1}e^{-u}\left(1-\left(1-\frac{2Tu}{r_0^2} \right)^{\lambda/2} \right)du + \int_{\frac{r_0^2}{2T}}^\infty u^{\frac{\lambda}{2}-1}e^{-u}du \right] =: I_1 + I_2.$$
	We use the expansion $ (1-x)^\alpha=\sum_{k=0}^\infty (\alpha \log(1-x))^k/k! $ for $ |x|<1 $ and $ \alpha\to 0 $ in order to analyze the first integral. This gives the following asymptotic equivalence
	\begin{align*}
		I_1 & \underset{\lambda \rightarrow 0^+}{=} 
		-\frac{\lambda^2}{4} \int_0^{\frac{r_0^2}{2T}} u^{-1}e^{-u}\log \left(1-\frac{2 Tu }{r_0^2}\right)du + o(\lambda^2) = O(\lambda^2).
	\end{align*}
	
	And on the other side:
	$$ I_2 = \frac{1}{\Gamma\left(\frac{\lambda}{2}\right)} \int_{\frac{r_0^2}{2T}}^\infty u^{\frac{\lambda}{2}-1}e^{-u}du \underset{\lambda \rightarrow 0^+}{=} \frac{\lambda}{2} \int_{\frac{r_0^2}{2T}}^\infty \frac{e^{-u}}{u} du \ + O(\lambda^2) .$$
	
	From here we obtain the asymptotic expansion \eqref{eq:exp1}.
	
	{\it Step 3: }
	We get \eqref{tail_bessel_integral} from the expansion \eqref{eq:exp1} and using Karamata Tauberian theorems (see \cite{embrechts1997}, Corollary A3.10):
	
	\begin{proposition}
	Let $df$ be a measure on $\mathbb{R}^+$, $F(x) := \int_0^x df$, $\hat{f}(\lambda) := \int_0^\infty e^{-\lambda x}F(x) dx$, $0 \le \gamma < 1$ and let $L>0$. Then the following are equivalent:
	\begin{enumerate}
		\item $ \displaystyle 1 - \hat{f}(\lambda) \sim L \lambda^\gamma , \ \ \lambda \rightarrow 0^+ $
		\item $ \displaystyle 1 - F(x) \sim \frac{L}{\Gamma(1-\gamma)}x^{-\gamma}, \ \ x \rightarrow \infty.$
	\end{enumerate}
	\end{proposition}

\end{proof}

Next we will give an estimate for $\mathbb{E}[F(T\wedge \zeta_\varepsilon)^p]$ in the case of the modified algorithm for reflected Brownian motion.

\begin{lemma}
	\label{majoration_bessel_expectation}
We define the stopping time $\zeta_\varepsilon := \inf \lbrace t \in [0,T]: \ R_t^2/(T-t) \le \varepsilon \rbrace$. Then we have the following bound for $ p\in (1,2) $ and $\varepsilon > 0$ small enough:
\begin{equation}
\mathbb{E}^{r_0} \left[F(T\wedge \zeta_\varepsilon)^p\right] \le C\left(1+ \log\left(\frac{2T}{r_0^2}\right) \right )\left({\varepsilon T} \right)^{1-p}.
\end{equation}
\end{lemma}
\begin{proof}
By \cite[(6.2.3)]{jeanblanc2010}, the density of $R_t$ in $r$ is
\begin{equation}
\label{eq:bessel_density}
\frac{r}{t} \left(\frac{r}{r_0}\right)^\nu e^{-\frac{r^2 + r_0^2}{2t}} I_\nu\left( \frac{rr_0}{t} \right),
\end{equation}
where $\nu = \frac{d}{2}-1$ and $d$ is the dimension of the Bessel process. Taking $d=2$, we have:
\begin{align*}
& \mathbb{E}^{r_0} \left[F(T\wedge \zeta_\varepsilon) ^p\right] \le \int_0^T \mathbb{E}^{r_0} \left[ \frac{1}{R_t^{2p}} \mathds{1}_{R_t \ge \sqrt{\varepsilon (T-t)}} \right] dt = \int_0^T \int_{\sqrt{\varepsilon (T-t)}}^\infty \frac{1}{r^{2p-1}t}e^{-\frac{r^2+r_0^2}{2t}} I_0\left( \frac{rr_0}{t} \right) dr dt\\
& \quad = \int_{\sqrt{\varepsilon T}}^\infty \int_0^T \frac{1}{r^{2p-1}t}e^{-\frac{r^2+r_0^2}{2t}} I_0\left( \frac{rr_0}{t} \right) dt dr + \int_0^{\sqrt{\varepsilon  T}} \int_{T-r^2/\varepsilon}^T \frac{1}{r^{2p-1}t}e^{-\frac{r^2+r_0^2}{2t}} I_0\left( \frac{rr_0}{t} \right) dt dr =: I_1 + I_2 .
\end{align*}
Then, using \eqref{majoration_I0} for both $I_1$ and $I_2$:
\begin{align*}
I_1 \le& \int_{\sqrt{\varepsilon T}}^\infty \frac{1}{r^{2p-1}} \int_0^T \frac{1}{t}e^{-\frac{(r-r_0)^2}{2t}} dt dr =\int_{\sqrt{\varepsilon T}}^\infty \frac{1}{r^{2p-1}} \int_{\frac{(r-r_0)^2}{2T}}^\infty \frac{e^{-v}}{v} dv dr\\
= & \int_{\sqrt{\varepsilon T}}^{\sqrt{2T}+r_0} \frac{1}{r^{2p-1}} \int_{\frac{(r-r_0)^2}{2T}}^\infty \frac{e^{-v}}{v} dv dr + \int_{\sqrt{2T}+r_0}^\infty \frac{1}{r^{2p-1}} \int_{\frac{(r-r_0)^2}{2T}}^\infty \frac{e^{-v}}{v} dv dr.
\end{align*}
But we have, for all $a \ge 0$:
$$ \int_a^\infty \frac{e^{-x}}{x} dx \le \left\lbrace
\begin{array}{l}
e^{-a} \ \text{ if } a \ge 1, \\
e^{-1} + \log(1/a) \ \text{ if } a < 1.
\end{array} \right. $$
So that
$$ I_1 \le \int_{\sqrt{\varepsilon T}}^{\sqrt{2T}+r_0} \frac{dr}{r^{2p-1}}\left(e^{-1} + \log\left(\frac{2T}{(r-r_0)^2}\right) \right) + \int_{\sqrt{2T}+r_0}^\infty \frac{dr}{r^{2p-1}}e^{-\frac{(r-r_0)^2}{2T}} \underset{\varepsilon \rightarrow 0}{\sim} \frac{(\varepsilon T)^{1-p}}{2(p-1)}\left(e^{-1} + \log\left(\frac{2T}{r_0^2}\right)\right) . $$

For the second integral, we have:
\begin{align*}
I_2 & \le \int_0^{\sqrt{\varepsilon T}}  \frac{1}{r^{2p-1}} \int_{T-r^2/\varepsilon}^T e^{-\frac{(r-r_0)^2}{2t}} \frac{dt}{t} dr = \int_0^{\sqrt{\varepsilon  T}} \frac{1}{r^{2p-1}} \int_{\frac{(r-r_0)^2}{2T}}^{\frac{(r-r_0)^2}{2(T-r^2/\varepsilon  )}} \frac{e^{-v}}{v} dv dr \le \int_0^{\sqrt{\varepsilon T}} \frac{1}{r^{2p-1}} \int_{\frac{(r-r_0)^2}{2T}}^{\frac{(r-r_0)^2}{2(T-r^2/\varepsilon  )}} \frac{dv}{v} dr \\
& = -\int_0^{\sqrt{\varepsilon T}} \frac{1}{r^{2p-1}} \log\left(1-\frac{r^2}{\varepsilon  T}\right) dr = - \frac{1}{(\varepsilon T)^{p-1}}\int_0^1 \frac{1}{u^{2p-1}} \log(1-u^2) du,
\end{align*}
where the last integral converges as soon as $p<2$.
\end{proof}

\section{Proofs in Section \ref{section:density_formulas}}
\label{appendix:B}

\begin{proof}[Proof of Theorem \ref{Rformula_alpha_pi_m} in the case that $ \mathcal{D} =\langle \pi /m \rangle $]
	We apply the method of images in $\mathbb{R}^2$. Let $f : \mathcal{D} \rightarrow [0,+\infty)$ be a non-negative continuous function with compact support. The heat equation with boundary conditions
	\begin{equation}
		\label{reflected_pde}
		\begin{array}{ll}
			\partial_t u (t,x) = \frac{1}{2} \Delta u (t,x), \ & (t,x)\in [0,+\infty) \times \mathcal{D}  \\
			u(0,x) = f(x), \ & x\in \mathcal{D}\\
			\nabla u(t,x) \cdot n(x) = 0, \ & t > 0, \ x \in \partial \mathcal{D} , 
		\end{array}
	\end{equation}
	where $n(x)$ denotes the inward unitary orthogonal vector on the boundary,
	can be rewritten as a partial differential equation in polar coordinates $ (r,\theta) $ as described in \cite[Chapter 1]{Mazja}. Furthermore if we perform the change of variable $ z=\log(r) $, we obtain a parabolic problem with non-constant coefficients in a strip with mixed type boundary condition. The existence and uniqueness for this PDE is treated in \cite[chapter IV, Theorem 5.3]{Lady}. From this reference and the Feyman-Kac representation theorem we obtain that
	$$ u(t,x) := \mathbb{E}^{x}\left[f(X_{t})\right]$$ satisfies the heat equation on $\mathcal{D}$ described above.
	Now we define the function
	$$ \begin{array}{ll}
		\tilde{f} & : \mathbb{R}^2 \rightarrow \mathbb{R} \\
		y & \mapsto f(T_k^{-1}y) \ \text{ for } y \in \mathcal{D}_k,
	\end{array} $$
	and for $t \ge 0$ and $x \in \mathbb{R}^2$,
	$$ \tilde{u}(t,x) := \mathbb{E}^x\left[ \tilde{f}(W_t) \right] .$$
	Then $\tilde{u}$ satisfies the heat equation on $\mathbb{R}^2$ with $\tilde{f}$ as initial condition, so that
	$$ \tilde{u}(t,x) = \frac{1}{2\pi t} \int_{\mathbb{R}^2} \tilde{f}(y) e^{-\frac{|x-y|^2}{2t}} dy = \frac{1}{2\pi t} \sum_{k=0}^{2m-1} \int_{\mathcal{D}_k} \tilde{f}(y) e^{- \frac{|x-y|^2}{2t}} dy = \frac{1}{2\pi t}\int_{\mathcal{D}} f(y) \sum_{k=0}^{2m-1} e^{- \frac{|x-T_k y|^2}{2t}} dy. $$
	
	On the other hand, $u$ and $\tilde{u}$ satisfy the same boundary equation on $\mathcal{D}$. Indeed, for all $x \in \mathcal{D}$, $\tilde{f}(x) = f(x)$ so $u$ and $\tilde{u}$ satisfy the same initial conditions. For $x \in \partial\mathcal{D}^-$, $k=0, \ \ldots, \ 2m-1$ and $y \in \mathcal{D}$, we have
	\begin{align*}
		(x-T_k y)\cdot n(x) & = - (x-T_{_{2m-1-k}}y) \cdot n(x) , \\
		|x-T_k y|^2 & = |x-T_{_{2m-1-k}}y|^2 ,
	\end{align*}
	so that
	\begin{align*}
		\nabla \tilde{u}(t,x) \cdot n(x) =& - \frac{1}{2\pi t^2}\int_{\mathcal{D}} f(y) \sum_{k=0}^{2m-1} (x-T_ky)\cdot n(x)e^{- \frac{|x-T_k y|^2}{2t}} dy\\
		= & - \frac{1}{2 \pi t^2} \int_{\mathcal{D}} f(y) \sum_{k=0}^{m-1} (x-T_k y) \cdot n(x) e^{- \frac{|x-T_k y|^2}{2t}} dy\\& + \frac{1}{2 \pi t^2} \int_{\mathcal{D}} f(y) \sum_{k=0}^{m-1} (x - T_{_{2m-1-k}}y)\cdot n(x)  e^{- \frac{|x-T_{_{2m-1-k}} y|^2}{2t}} dy = 0.
	\end{align*}
	If $x \in \partial D^+$, we get the same result noting that for all $y \in \mathcal{D}$ and $k=0, \ \ldots, \ 2m-1$, 
	\begin{align*}
		(x-T_{_{(k+1) \text{ mod } 2m}} y)\cdot n(x) & = - (x-T_{_{(2m-k) \text{ mod } 2m}}y) \cdot n(x) , \\
		|x-T_{_{(k+1) \text{ mod } 2m}} y|^2 & = |x-T_{_{(2m-k) \text{ mod } 2m}}y|^2 ,
	\end{align*}
	So $u(t,x) = \tilde{u}(t,x)$ for all $t \ge 0$ and $x \in \mathcal{D}$, and
	$$ \mathbb{E}^x\left[f(X_t)\right] = \frac{1}{2\pi t}\int_{\mathcal{D}} f(y) \sum_{k=0}^{2m-1} e^{- \frac{|x-T_k y|^2}{2t}} dy .$$
\end{proof}

\begin{proof}[Proof of Theorem \ref{Rformula_alpha_pi_m} in the general case]

	As in the above proof we consider
	\begin{equation}
		\label{feynman_kac_representation}
		u(t,x) = \mathbb{E}^x\left[ f(X_t) \right] ,
	\end{equation}
	where $t \ge 0$, $x \in \mathcal{D}$, and $f : \mathcal{D} \rightarrow [0,+\infty)$ is a non-negative continuous function with compact support. Then $u$ is the solution of the partial differential equation \eqref{reflected_pde}.
	Considering the formula obtained for the case $\alpha = \pi/m$ above, we would like to express it so that it does not depend explicitly on $m$. We first assume that $\alpha = \pi/m$ for some $m \in \mathbb{N}$. In order to switch to polar coordinates, let $x = (r_0 \cos(\theta_0), r_0 \sin(\theta_0))$, $y = (r \cos(\theta), r \sin(\theta))$ and recall that $\vartheta_k$ for $0 \le k \le 2m-1$ is the angle of $T_k(y)$ in $[0,2\pi)$, i.e.
	$T_k(y) = (r \cos(\vartheta_k), r \sin(\vartheta_k))$ with $\vartheta_{2k} = 2k \alpha + \theta$ and $\vartheta_{2k+1} = 2(k+1) \alpha - \theta$.
	We then rewrite \eqref{alpha_pi_m_reflected}:
	\begin{align}
		\mathbb{P}^{x} (X_t \in dy) = \frac{r}{2\pi t} e^{-\frac{r^2 + r_0^2}{2t}} \sum_{k=0}^{2m-1} e^{\frac{rr_0}{t}\cos(\theta_0 - \vartheta_k)} dr d\theta.
		\label{eq:theta}
	\end{align}
	We use the following identity
	(see \cite[page 933, (8.511.4)]{Grad}), valid for $\gamma, z \ge 0 $:
	$$  \ e^{\gamma z} =I_0(z)+ 2 \sum_{n=1}^\infty T_n(\gamma) I_n(z), $$
	where $T_n$ is the $n^{\text{th}}$ Tchebychev's polynomial of the first kind and $I_n$ the modified Bessel function of the first kind with order $n$. Then the result \eqref{formula_stopped} follows because
	$ T_n(\cos(\theta)) = \cos(n\theta)$ and
	\begin{align*}
		\sum_{k=0}^{2m-1} \cos(n(\theta_0-\vartheta_k)) =
		\begin{cases}
			2m \cos(n \theta) \cos(n\theta_0)& \text{ if } n \text{ is a multiple of } m,\\
			0&\text{otherwise.}
		\end{cases}
	\end{align*}
	Using the above properties and $m = \pi/\alpha$, we obtain
	\begin{align}
	\label{eq:density_reflected_formula}
		\mathbb{P}^{x}(X_t \in dy) =& \frac{2r}{t \alpha} e^{-(r^2 + r_0^2)/(2t)}\left( \frac{1}{2}I_0\left(\frac{rr_0}{t}\right) + \sum_{n=1}^{\infty} I_{n\pi / \alpha}\left(\frac{rr_0}{t}\right) \cos\left(\frac{n\pi \theta}{\alpha}\right) \cos\left(\frac{n \pi \theta_0}{\alpha}\right) \right) dr d\theta ,
	\end{align}
	which gives a formula which does not depend on $m$. Now, we have to check that this formula is well defined for any $\alpha \in (0,\pi)$ and that it is the unique solution the partial differential equation \eqref{reflected_pde} with $ \mathcal{D}=\langle \alpha\rangle $.

	This is explained in Corollary \ref{cor:2}. In our current situation, $ d=2 $, the eigenvalues are $ \lambda_j:=\left(j\pi/\alpha\right)^2 $ for $ j \ge 0$ and the eigenfunctions are $ m_j(\theta)=\sqrt{2/\alpha}\cos\left(j\pi\theta/\alpha\right) $, $ j\geq 1 $ and $ m_0\equiv 1/\sqrt{\alpha} $.
	\end{proof}
	
	\begin{corollary}
		\label{cor:2}Consider a general $ d $-dimensional cone generated by all rays emanating from the origin and passing through a compact subset $ D\subset \mathbb{S}^{d-1} $ which has smooth boundary.  Consider $ X $ to be the normally reflected Brownian motion at the boundary of the cone. Then $ X_t $ has a density given by
		\begin{align}
			\label{Gformula_reflected}
			& \mathbb{P}^{x}(X_t \in dy) = \frac{r}{t (rr_0)^{d/2-1}} e^{-\frac{r^2 + r_0^2}{2t}} \left( I_{\alpha_0}\left(\frac{rr_0}{t}\right) m_0(\theta)m_0(\theta_0)+ \sum_{n=1}^{\infty} I_{\alpha_n}\left(\frac{rr_0}{t}\right) m_n(\theta)m_n(\theta_0) \right)dr d\theta . 
		\end{align}
	
	To explain the elements in the above formula, denote by 
	$L_{\mathcal{S}^{d-1}}$, the Laplace-Beltrami operator on $\mathbb{S}^{d-1}$. With the above assumptions, there exists  a complete set of orthonormal eigenfunctions $m_{j}$ with corresponding eigenvalues $0\leq \lambda_{0}<\lambda_{1} \leq \lambda_{2}<\ldots$ satisfying 
	$$
	\left\{\begin{array}{ll}
		L_{\mathcal{S}^{d-1}} m_{j}(x)=-\lambda_{j} m_{j}(x) & \text { for }x\in D \\
		\nabla m_{j}(x)\cdot n(x)=0 & \text { for } x\in \partial D ,
	\end{array}\right.
	$$
	$$
	\alpha_{j}=\sqrt{\lambda_{j}+\left(\frac{d}{2}-1\right)^{2}}.
	$$
		\end{corollary}
	\begin{proof}
		The beginning of the proof is the same as in the statements following \eqref{reflected_pde} in what refers to the existence and uniqueness of the associated PDE. Existence and uniqueness for the reflected process in the generalized cone can be deduced from \cite{bass1996} (see also \cite[Remark 4.1]{bass2005} and the references therein).
		
			In order to prove that \eqref{Gformula_reflected} satisfies the associated PDE, one follows a similar proof in \cite{Ban} for the killed case.	
		The proof in the reflected case follows line by line, the proof in \cite{Ban}, except that  the required estimates and properties of the eigenvalues and eigenfunctions of the Laplace-Beltrami operator in the case of Neumann boundary conditions have to be referred to the proper literature (for this see, \cite{greben}, \cite{Grieser} and \cite{kroger}) although the estimates do not change as far as it relates to the proof of \cite[Lemma 1]{Ban} with the corresponding corrections of typos. In order for the required estimates to be satisfied one needs that the generalized $ d$-dimensional cone is generated by all rays emanating from the origin and passing through a compact subset $ D\subset \mathbb{S}^{d-1} $ which has smooth boundary. 
		
		We also remark that the density expressions in \cite{Ban} are written under polar measures which explains why our expressions have an extra $ r $ which appears due to the Jacobian of the change of coordinates.
		\end{proof}

We also provide a full elementary proof that \eqref{eq:density_reflected_formula} satisfies the initial conditions in the Supplementary Material.

\section{A hint for higher dimensions}
\label{sec:higher_dim}

The following proposition is provided so as to hint at the possibilities in dimension higher than two. It shows that dimension two is the case which is mathematically difficult to treat.
\begin{proposition}
	\label{prop:15}
	Let $(R_t)_{t\geq 0}$ be a Bessel process in dimension $d \ge 3$ and let $F(T) = \int_0^t ds/R_s^2$. Then:
	$$ \mathbb{E}^{r_0}\left[ F(T)\right] < + \infty.$$
\end{proposition}
\begin{proof}
	The density of $R_t$ in $r$ is given in \eqref{eq:bessel_density}, so by performing the change of variable $r = ut$, we have
	$$ \mathbb{E}^{r_0}\left[ F(T)\right] = \int_0^T \int_0^\infty \frac{1}{rt} \left( \frac{r}{r_0} \right)^\nu e^{-\frac{r^2+r_0^2}{2t}} I_\nu\left( \frac{rr_0}{t} \right) dr dt  = \frac{1}{r_0^\nu} \int_0^T t^{\nu-1} \int_0^\infty I_\nu(ur_0)e^{-\frac{tu^2}{2} - \frac{r_0^2}{2t}}u^{\nu-1}du dt .$$
	Note that for all $\varepsilon > 0$,
	$$ \int_\varepsilon^T t^{\nu-1} \int_0^\infty I_\nu(ur_0)e^{-\frac{tu^2}{2} - \frac{r_0^2}{2t}}u^{\nu-1}du dt \le \left(\int_\varepsilon^T e^{-\frac{r_0^2}{2t}}t^{\nu-1}dt\right) \left( \int_0^\infty I_\nu(ur_0)e^{-\varepsilon u^2/2}u^{\nu-1}du\right) < \infty, $$
	where for we used the Proposition \ref{majoration_In} for the convergence of the second integral. So to prove the convergence of the integral, we only need to prove the convergence for the integral in $ t $ around zero. To do so we use Proposition \ref{majoration_In} again so that for $ \nu, \ r_0>0 $:
	$$ \mathbb{E}^{r_0}\left[ F(\varepsilon) \right] \le \frac{1}{r_0^\nu} \int_0^\varepsilon t^{\nu-1}e^{- \frac{r_0^2}{2t}} \int_0^\infty \left(e^{ur_0} + \frac{1}{\pi(\nu + ur_0)} \right) e^{-\frac{tu^2}{2} }u^{\nu-1}du dt <\infty.$$
	In fact, using that $\nu > 0$, we have
	\begin{align*} 
		I_1 & := \frac{1}{r_0^\nu} \int_0^\varepsilon t^{\nu-1} \int_0^\infty e^{-\frac{t}{2}\left(u-\frac{r_0}{t}\right)^2} u^{\nu-1} du dt \le \frac{1}{r_0^\nu} \int_0^\varepsilon t^{\nu-1} \int_{-\infty}^\infty e^{-\frac{t}{2}u^2} \left( |u| +\frac{r_0}{t} \right)^{\nu-1} du dt \\
		& = \frac{2}{r_0^\nu} \int_0^\varepsilon t^{\nu-3/2} \int_{0}^\infty e^{-u^2/2} \left(\frac{u}{\sqrt{t}} + \frac{r_0}{t} \right)^{\nu-1} du dt \le \frac{2}{r_0^\nu} \int_0^\varepsilon t^{\nu-3/2} \int_{0}^\infty e^{-u^2/2} \left(\frac{u+ r_0}{t} \right)^{\nu-1} du dt \\
		& = \frac{2}{r_0^\nu} \int_0^\varepsilon t^{-1/2} \int_{0}^\infty e^{-u^2/2} \left({u+ r_0} \right)^{\nu-1} du dt < \infty.
	\end{align*}
	As for the second integral, we have
	$$ I_2 := \frac{1}{r_0^\nu} \int_0^\varepsilon t^{\nu-1} \int_0^\infty \frac{1}{\pi(\nu + ur_0)} e^{-\frac{tu^2}{2} - \frac{r_0^2}{2t}}u^{\nu-1}du dt \le \frac{1}{r_0^\nu} \int_0^\varepsilon t^{\nu-1}e^{-\frac{r_0^2}{2t}} dt \int_0^\infty \frac{u^{\nu-1}}{\pi(\nu + ur_0)} du < \infty .$$
\end{proof}

\section*{Acknowledgements}
The first author thanks \'Ecole Normale Sup\'erieure, D\'epartement de Math\'ematiques et Applications for providing a financial support for a visit to Ritsumeikan University.

The second author is supported in part by KAKENHI 20K03666.

Both authors thank the anonymous referee for a careful review and for correcting mistakes and providing references.

\newpage

\section{Supplementary material}

\subsection{Proof of the convergence to the initial condition when $t \rightarrow 0$}

In this section, we prove the following result:

\begin{proposition}
\label{initial_conditions_proof}
The formula \eqref{formula_reflected} satisfies the initial conditions of the heat equation, i.e. for all functions $f : \mathcal{D} \rightarrow \mathbb{R}$ continuous with compact support, $r_0>0$ and $\theta_0 \in (0,\alpha)$:
\begin{equation}
\label{initial_conditions_convergence}
\int_{\mathcal{D}} \hat{f}(r,\theta) \frac{2r}{t \alpha} e^{-(r^2 + r_0^2)/2t} \left( \frac{1}{2}I_0\left(\frac{rr_0}{t}\right) + \sum_{n=1}^{\infty} I_{n\pi / \alpha}\left(\frac{rr_0}{t}\right) \cos\left(\frac{n\pi \theta}{\alpha}\right) \cos\left(\frac{n \pi \theta_0}{\alpha}\right)\right) dr d\theta \underset{t \rightarrow 0}{\longrightarrow} \hat{f}(r_0,\theta_0),
\end{equation}
where $\hat{f}: [0,\infty) \times [0,\alpha]\to\mathbb{R}$ denotes the function $f$ expressed in polar coordinates.
\end{proposition}

\medskip

For the proof, we use the following representation: 
\begin{align*}
\forall x >0, \ \forall \alpha \ge 0, \ I_\alpha(x) = \frac{1}{\pi} \int_0^\pi e^{x\cos(u)}\cos(\alpha u)du - \frac{\sin(\alpha \pi)}{\pi} \int_0^\infty e^{-x \cosh(u) - \alpha u} du .
\end{align*}
Let us denote by $A_1$ and $A_2$ the two terms appearing in the integral representation in \eqref{initial_conditions_convergence} after replacing with the above formula, with $d_n=\frac{1}{2}$ if $n=0$ and $d_n=1$ for $ n\geq 1 $:
\begin{align*}
A_1 & := \frac{2}{\alpha \pi} \int_{\mathcal{D}} \frac{r}{t} \hat{f}(r,\theta) e^{-\frac{r^2+r_0^2}{2t}} \sum_{n=0}^{\infty} d_n \cos\left(\frac{n\pi\theta}{\alpha}\right) \cos\left(\frac{n\pi\theta_0}{\alpha}\right) \int_{0}^\pi e^{\frac{rr_0}{t}\cos(u)}\cos\left(\frac{n\pi}{\alpha}u\right)du dr d\theta, \\
A_2 & := -\frac{2}{\alpha \pi}\int_{\mathcal{D}} \frac{r}{t} \hat{f}(r,\theta) e^{-\frac{r^2 + r_0^2}{2t}} \sum_{n=1}^\infty \sin\left(\frac{n\pi^2}{\alpha}\right) \cos\left(\frac{n\pi\theta}{\alpha}\right) \cos\left(\frac{n\pi\theta_0}{\alpha}\right) \int_0^\infty e^{-\frac{rr_0}{t}\cosh(u) - \frac{n\pi}{\alpha}u}du dr d\theta.
\end{align*}

\medskip

\begin{proposition}
	We have $A_2 \underset{t \rightarrow 0}{\longrightarrow} 0 $.
\end{proposition}

\begin{proof}
	{\it Step 1:} For the moment let us assume that the integration order can be switched. For all $a\in\mathbb{R}$ and $b >0$, we have:
	\begin{equation}
	\label{complex_formula}
	g(a,b) := \sum_{n=0}^\infty \sin(na)e^{-nb} = \Im\left( \sum_{n=0}^\infty e^{ina-nb}\right) = \frac{e^{-b}\sin(a)}{\left(1-e^{-b}\cos(a)\right)^2+e^{-2b}\sin^2(a)} .
	\end{equation}
	We use the following trigonometric identity written in compact form
	\begin{align}
	\label{trigonometric_formula}
	4\cos(a)\cos(b)\sin(c)=&\sin(a+b+c)+\sin(c-a-b)+\sin(a-b+c)+\sin(c-a+b)\\
	A_{jk}(\theta):=&\frac{\pi}{\alpha}((-1)^j\theta +(-1)^k \theta_0 + \pi), \nonumber
	\end{align} 
	$$ \sum_{n=1}^\infty \cos\left(\frac{n\pi\theta}{\alpha}\right) \cos\left(\frac{n\pi\theta_0}{\alpha}\right) \sin\left(\frac{n\pi^2}{\alpha}\right) e^{-\frac{n\pi}{\alpha}u} = \frac{1}{4}\sum_{j,k=0}^1 g\left(A_{jk}(\theta),\frac{\pi}{\alpha}u\right). $$
	
	Now we prove that for all $j,k \in \{0,1\}$, we have:
	$$ J_{j,k} := \frac{1}{t} \int_{\mathcal{D}} r \hat{f}(r,\theta) e^{-\frac{r^2+r_0^2}{2t}}\int_0^\infty e^{-\frac{rr_0}{t}\cosh(u)} g\left(A_{jk}(\theta),\frac{\pi}{\alpha}u\right) du d\theta dr \underset{t \rightarrow 0}{\longrightarrow} 0.$$
	We will do the analysis of $ J_{j,k} $ dividing the integration region in two: $ J_{j,k}^1 $ which comprises the integral on the region $ \mathcal{D}\times (1,\infty) $ and the remaining which is denoted by $ J_{j,k}^2 $.
	With this in mind, note that $g$ is continuous on $ [0,\infty) \times (0,\alpha)$ and is locally integrable in $(0,0)$. In fact, for all $a \in \mathbb{R}$ and $b >0$:
\begin{align}
\label{eq:21.1}
	g(a,b) = \frac{e^{-b}\sin(a)}{\left(1-e^{-b}\cos(a)\right)^2+e^{-2b}\sin^2(a)} \underset{a,b \rightarrow 0}{\sim} \frac{a}{\left((1-e^{-b}\left(1-\frac{a^2}{2}\right)\right)^2 + a^2} \sim \frac{a}{a^2 + b^2}.
	\end{align}
	Moreover,
	$$  \int_{[0,1]^2}\frac{a}{a^2 + b^2} da db = \int_{0}^1 \left[ \frac{1}{2}\log(a^2 + b^2) \right]_{a=0}^{a=1} db = \frac{1}{2}\int_0^1 \left(\log(1+b^2) - \log(b^2)\right) db < \infty .$$
	 Next, note that $\partial_b g(a,b)$ is negative, so
	$$ \forall u \ge 1, \ \forall \theta \in \mathbb{R}, \ \left|g\left(A_{jk}(\theta),\frac{\pi}{\alpha}u\right)\right| \le \frac{e^{-\frac{\pi}{\alpha}}}{(1-e^{\frac{\pi}{\alpha}})^2}, $$
	and since $u \mapsto u^2/\cosh(u)$ is non-negative and bounded above, there exists $B>0$ such that
\begin{align*}
|J^1_{jk}| \le \frac{C_1}{t} \int_{0}^\infty re^{-\frac{r^2+r_0^2}{2t}} \int_{1}^\infty e^{-B\frac{rr_0}{t}u^2} du dr \le \frac{C_2}{\sqrt{tr_0}} e^{-\frac{r_0^2}{2t}} (2t)^{3/4} \int_0^\infty \sqrt{r'}e^{-r'^2}dr' \underset{t \rightarrow 0}{\longrightarrow} 0 .
\end{align*}
	
	On the other hand, using \eqref{eq:21.1} and the continuity of $ g $ we obtain
	$$ |J^2_{jk}| \le ||f||_\infty \frac{1}{t} \int_{0}^\infty re^{-\frac{r^2+r_0^2}{2t}}dr \cdot \int_{0}^1\int_{0}^\alpha \left|g\left(A_{jk}(\theta),\frac{\pi}{\alpha}u \right) \right| d\theta du \underset{t \rightarrow 0}{\longrightarrow} 0. $$
	So that
	$$|A_2| \le \frac{1}{2 \alpha \pi} \sum_{j,k=0}^1 |J_{jk}| \underset{t \rightarrow 0}{\longrightarrow} 0. $$
	
	{\it Step 2:} Now, {we prove that one can interchange the order of the integrals and sum in $A_2$}. Note that Fubini's theorem does not apply here because of the factors $\cos(n\pi\theta/\alpha)$ and $\cos(n\pi\theta_0/\alpha)$. A computation similar to \eqref{complex_formula} leads to, for all $N \in \mathbb{N}$:
	$$ g_N(a,b) := \sum_{n=N}^\infty \sin(na)e^{-nb} = \Im \left( \sum_{n=N}^\infty e^{-nb + ina} \right) = e^{-Nb} \frac{\sin(Na) - e^{-b}\sin((N-1)a)}{(1-e^{-b}\cos(a))^2 + e^{-2b}\sin^2(a)}.$$
	
	Next, we prove that, for all $j,k \in \{0,1\}$ and for all $t>0$, denoting $R^N_{jk}$ the difference between the infinite sum and the partial sum up in $A_2$ to $N$,
	$$ R^N_{jk} := \frac{1}{t} \int_{\mathcal{D}} re^{-\frac{r^2+r_0^2}{2t}}\hat{f}(r,\theta) \int_0^\infty e^{-\frac{rr_0}{t}\cosh(u)} g_N\left(A_{jk}(\theta),\frac{\pi}{\alpha}u\right) du d\theta dr \underset{N \rightarrow \infty}{\longrightarrow} 0.$$
	
	Let us remark that
	$$ e^{Nb}g_N(a,b) = \frac{\sin(Na) - e^{-b}\sin((N-1)a)}{(1-e^{-b}\cos(a))^2 + e^{-2b}\sin^2(a)} \underset{a,b \rightarrow 0}{\sim} \frac{a}{a^2 + b^2},$$
	so that $(a,b) \mapsto e^{Nb}g_N(a,b)$ is integrable in $(0,0)$. We denote $R^{N,1}_{jk}$ and $R^{N,2}_{jk}$ the two terms obtained after splitting the integral with respect to $u$ on $(0,1)$ and $(1,\infty)$ respectively.
	We have then
	$$  |R^{N,2}_{jk}| \le \frac{1}{t} ||f||_\infty \int_{0}^\infty re^{-\frac{r^2+r_0^2}{2t}}dr \int_{0}^1 e^{-\frac{N\pi u}{\alpha}} \int_{0}^\alpha e^{\frac{N\pi u}{\alpha}} \left|g_N\left(A_{jk}(\theta),\frac{\pi}{\alpha}u\right)\right| d\theta du \underset{N \rightarrow \infty}{\longrightarrow}0, $$
	where we use Lemma \ref{integral_times_exponential} for the above convergence.
	Moreover for all $u \ge 1, \ \theta \in \mathbb{R}$,
	$$ |g_N(A_{jk}(\theta),\frac{\pi}{\alpha}u)| \le e^{-N\frac{\pi}{\alpha}u}\frac{2}{(1-e^{-\frac{\pi}{\alpha}})^2},$$
	so we have 
	$$  |R^{N,1}_{jk}| \le \frac{1}{t} ||f||_\infty \int_{0}^\infty re^{-\frac{r^2+r_0^2}{2t}}dr \int_{1}^\infty e^{-N\frac{\pi}{\alpha}u} \frac{2}{(1-e^{-\frac{\pi}{\alpha}})^2}du \underset{N \rightarrow \infty}{\longrightarrow} 0. $$
From the above arguments we obtain the conclusion:	$  R^N_{jk} \rightarrow 0$ for all $j,k \in \lbrace 0,1 \rbrace$.
	
\end{proof}

\begin{lemma}
	\label{integral_times_exponential}
	Let $b >0$ and let $f \in L^1((0,b))$ be non-negative and continuous. Then:
	$$ \int_0^b e^{-Nx}f(x)dx \underset{N \rightarrow \infty}{\longrightarrow} 0. $$
\end{lemma}
\begin{proof}
	Let $\varepsilon >0$ and choose $\delta >0$ such that $\int_0^\delta f(x)dx \le \varepsilon$. Then for $N$ big enough:
	$$ \int_0^b e^{-Nx}f(x)dx \le \int_0^\delta f(x)dx + e^{-N\delta} ||f||_1 \le 2\varepsilon. $$
\end{proof}

\medskip

\begin{proposition}
We have $A_1 \underset{t\rightarrow 0}{\longrightarrow} \hat{f}(r_0,\theta_0)$.
\end{proposition}

\begin{proof}

 First, let us assume that the integration order can be exchanged. This will be further discussed in Step 5.

\medskip

\textbf{Step 1 :} We exchange the order of integrals in  $A_1$ as:
\begin{align*}
A_1 & =  \int_{0}^\pi f_1(u,t) f_2(u,t)du,\\
f_1(u,t) & := \frac{2}{\sqrt{t}\alpha \pi} e^{-\frac{r_0^2\sin^2(u)}{2t}},\\
f_2(u,t):=&\int_{0}^\infty  \frac{re^{-\frac{(r-r_0\cos(u))^2}{2t}}}{\sqrt{t}} \sum_{n=0}^{\infty} d_n \cos\left(\frac{n\pi\theta_0}{\alpha}\right) \cos\left(\frac{n\pi}{\alpha}u\right) \int_{0}^\alpha \hat{f}(r,\theta) \cos\left(\frac{n\pi\theta}{\alpha}\right) d\theta dr.
\end{align*}

We will study the limit of $ A_1 $ in the above integral order. First, we treat the sum inside $ f_2(u,t) $ using trigonometric identities for $ \cos\left( \frac{n\pi(\theta_0 \pm  u)}{\alpha} \right)  $. We see that is enough to find the limit for:
\begin{align*}
& \frac{1}{2} \sum_{n=0}^\infty d_n \left( \cos\left( \frac{n\pi(\theta_0 + u)}{\alpha} \right) + \cos\left( \frac{n\pi(\theta_0 - u)}{\alpha} \right) \right) \int_{0}^\alpha \hat{f}(r, \theta) \cos\left(\frac{n\pi\theta}{\alpha}\right) d\theta.
\end{align*}

In our case, we will use the following normalization of the classical Fourier inversion formula:
\begin{theorem}[Fourier Formula]
\label{fourier}
Let $g : \mathbb{R} \rightarrow \mathbb{R}$ be a periodic and continuous function of period $2L$. Then the following series converges for all $x\in \mathbb{R}$ and:
$$ g(x) = \frac{a_0}{2} + \sum_{n=1}^\infty a_n \cos\left(\frac{n\pi x}{L}\right) + b_n \sin\left(\frac{n\pi x}{L}\right),$$
$$ \text{with } \ a_n = \frac{1}{L}\int_{-L}^L g(\theta)\cos\left(\frac{n\pi \theta}{L}\right)d\theta \ \text{ and } \ b_n = \frac{1}{L} \int_{-L}^L g(\theta) \sin \left( \frac{n\pi \theta}{L}\right) d\theta.$$
\end{theorem}

\smallskip

Next, we extend the definition of the function $\hat{f}$. For all $r \ge 0$ and $\theta \in [0,\alpha]$, we define $\hat{f}(r,-\theta) := \hat{f}(r,\theta)$, and we make it $2\alpha$-periodic by defining $\hat{f}(r,\theta + 2k\alpha) = \hat{f}(r,\theta)$. This way, $\hat{f}$ is an even and $2\alpha$-periodic function and is still continuous. With this definition, and using the Fourier inversion formula, we get for $ \theta_0\pm u\in (-\alpha,\alpha) $:
 \begin{align}
 \label{fourier_f}
  & \frac{1}{2} \sum_{n=0}^\infty d_n \cos\left( \frac{n\pi(\theta_0 \pm u)}{\alpha} \right) \int_{0}^\alpha \hat{f}(r, \theta) \cos\left(\frac{n\pi\theta}{\alpha}\right) d\theta \nonumber \\
  & = \frac{1}{4} \sum_{n=0}^\infty d_n \cos\left( \frac{n\pi(\theta_0 \pm u)}{\alpha} \right) \int_{-\alpha}^\alpha \hat{f}(r, \theta) \cos\left(\frac{n\pi\theta}{\alpha}\right) d\theta = \frac{\alpha}{4} \hat{f}(r,\theta_0 \pm u).
 \end{align}

\medskip

\textbf{Step 2 :} Next, let us study $ f_2(u,t) $ in two separate cases the integral in $r$: $$\int_{0}^\infty \hat{f}(r,\theta_0 \pm u) \frac{re^{-\frac{(r-r_0\cos(u))^2}{2t}}}{\sqrt{t}} dr .$$

{\it Case 1: }If $r_0\cos(u) > 0$:
$$ \frac{re^{-\frac{(r-r_0\cos(u))^2}{2t}}}{\sqrt{t}} = \frac{1}{\sqrt{t}} (r-r_0\cos(u))e^{-\frac{(r-r_0\cos(u))^2}{2t}} + r_0 \cos(u) \frac{e^{-\frac{(r-r_0\cos(u))^2}{2t}}}{\sqrt{t}}. $$
The total mass of the first term is
$$
\int _0^\infty \frac{1}{\sqrt{t}} (r-r_0\cos(u))e^{-\frac{(r-r_0\cos(u))^2}{2t}}dr=\left[-t\frac{e^{-\frac{(r-r_0\cos(u))^2}{2t}}}{\sqrt{t}}\right]_{r=0}^\infty = \sqrt{t} e^{-\frac{r_0^2\cos^2(u)}{2t}} \underset{t \rightarrow 0}{\longrightarrow} 0 .$$
And the second term is, up to the multiplicative constant $\sqrt{2\pi}$, an approximation of the unity around $r = r_0 \cos(u)>0$, so that in this case, the integral in $r$ converges to
$$\sqrt{2\pi} r_0 \cos(u) \hat{f}(r_0\cos(u),\theta_0 \pm u).$$
{\it Case 2:} If $r_0\cos(u) \le 0$: The total mass of $\frac{1}{\sqrt{t}} re^{-(r-r_0\cos(u))^2/(2t)} $ is bounded above by
$$ \frac{1}{\sqrt{t}} \int_0^\infty re^{-\frac{r^2}{2t}}dr = \sqrt{t} \underset{t \rightarrow 0}{\longrightarrow} 0,$$
so that the integral in $r$ converges to $0$. We remark here that convergences in the above two cases are uniform with respect to $ u $ within their respective domains.

\medskip

\textbf{Step 3 :} From the previous step, we consider now the integral with respect to $ u $. Taking into consideration the previous step, we can restrict to the case $ \cos(u)>0 $ or equivalently $ u\in [0,\pi/2] $. That is, consider 
$$ I_{\pm} := \frac{1}{\sqrt{t}} \int_{0}^{\pi/2} \hat{f}(r_0 \cos(u), \theta_0 \pm u) e^{-\frac{r_0^2\sin^2(u)}{2t}} \cos(u) du .$$
Note that for all $\varepsilon > 0$, by dominated convergence,
$$ \frac{1}{\sqrt{t}} \int_{\varepsilon}^{\pi/2} \hat{f}(r_0 \cos(u), \theta_0 \pm u) e^{-\frac{r_0^2\sin^2(u)}{2t}} \cos(u) du \underset{t \rightarrow 0}{\longrightarrow} 0.$$
Fix $\delta >0$ and take $\varepsilon$ small enough such that $\cos(u) \simeq 1$ and $\sin(u) \simeq u$ for all $u \in [0,\varepsilon]$. Then
$$ \frac{1}{\sqrt{t}} \int_0^\varepsilon e^{-\frac{r_0^2 \sin^2(u)}{2t}} \cos(u) du \underset{t \rightarrow 0}{\simeq} \frac{1}{\sqrt{t}} \int_0^\varepsilon e^{-\frac{r_0^2 u^2}{2t}} du \underset{t \rightarrow 0}{\longrightarrow} \sqrt{\frac{\pi}{2r_0^2}} .$$
Thus, up to the multiplicative constant $\sqrt{\pi/(2r_0^2)}$, $u \mapsto e^{-r_0^2\sin^2(u)/(2t)} \cos(u)$ is an approximation of the unity around $u=0$, so that
$$ I_{+}+I_- \underset{t \rightarrow 0}{\longrightarrow}  \sqrt{\frac{2\pi}{r_0^2}} \hat{f}(r_0,\theta_0).$$

\medskip

\textbf{Step 4 :} Now, we put all previous steps together. In Step 1, we proved that for all $t>0$:
\begin{align*}
f_2(u,t) & := \int_{0}^\infty \frac{re^{-\frac{(r-r_0\cos(u))^2}{2t}}}{\sqrt{t}} \frac{\alpha}{4} (\hat{f}(r,\theta_0 + u)+\hat{f}(r,\theta_0 - u)) dr .
\end{align*}
We have proved in Step 2 that for all $u$, $f_2(u,t)$ converges when $t \rightarrow 0$ to 
$$\overline{f}_2(u) := \frac{\alpha}{4} \sqrt{2{\pi}} r_0 \cos(u) \left(\hat{f}(r_0 \cos(u), \theta_0 + u)+\hat{f}(r_0 \cos(u), \theta_0 - u)\right),$$
and in Step 3 that $\int_0^\pi f_1(u,t) \overline{f}_2(u) du$ converges when $t \rightarrow 0$ to 
\begin{align*}
 &\underset{t \rightarrow 0}{\lim} \ \frac{2}{\alpha\pi \sqrt{t}} \int_0^{\pi/2} \frac{\alpha}{4} \sqrt{2{\pi}} r_0 
(\hat{f}(r_0 \cos(u), \theta_0 + u)+\hat{f}(r_0 \cos(u), \theta_0 - u))
e^{-\frac{r_0^2 \sin^2(u)}{2t}}\cos(u) du\\
 =& \frac{2}{\alpha\pi} \frac{\alpha}{4} \sqrt{2{\pi}} r_0  \sqrt{\frac{\pi}{2r_0^2}} 2\hat{f}(r_0,\theta_0) = \hat{f}(r_0,\theta_0).
\end{align*}
To end the proof of the convergence, we have to show that
\begin{align*}
\lim_{t \rightarrow 0}\int_0^\pi f_1(u,t) f_2(u,t) du=\lim_{t \rightarrow 0} \ \int_0^\pi f_1(u,t) \overline{f}_2(u) du.
\end{align*}
We have:
$$ \left| \int_0^\pi f_1(u,t)f_2(u,t)du - f_1(u,t)\overline{f}_2(u)du \right| \le \left(\sup_{t} \int_0^\pi f_1(u,t)du \right) ||f_2(\cdot,t) - \overline{f}_2(\cdot)||_\infty, $$
and $||f_2(\cdot,t) - \overline{f}_2(\cdot)||_\infty \rightarrow 0$ since in Step 2, the convergence is uniform with respect to $u$.

\medskip

\textbf{Step 5 :} We now prove that for all $r_0, \ \theta_0, \ t$ fixed, the integration order can be switched. Note that for fixed $n$, by Fubini's theorem, the integration order in $r$, $\theta$ and $u$ can be switched, as
\begin{align*}
& \int_{0}^\infty \int_{0}^\alpha \left| \hat{f}(r,\theta) e^{-\frac{r^2+r_0^2}{2t}} \cos\left( \frac{n\pi\theta}{\alpha}\right) \cos\left( \frac{n\pi\theta_0}{\alpha}\right) \right| \int_{0}^\pi \left| e^{\frac{rr_0}{t}\cos(u)} \cos\left(\frac{n\pi u}{\alpha}\right) \right| du d\theta dr \\
& \le ||f||_\infty \alpha \pi \int_{0}^\infty e^{-\frac{r^2+r_0^2}{2t}} e^{\frac{rr_0}{t}} dr < \infty.
\end{align*}
So that the order of integration can be exchanged for every partial sum. Now, for $N \in \mathbb{N}$:
\begin{align*}
& \left| \int_{0}^\pi e^{-\frac{r_0^2\sin^2(u)}{2t}} \int_{0}^\infty re^{-\frac{(r-r_0\cos(u))^2}{2t}} \sum_{n=0}^N d_n \cos\left(\frac{n\pi\theta_0}{\alpha}\right) \cos\left(\frac{n\pi}{\alpha}u\right) \cdot \int_{0}^\alpha \hat{f}(r,\theta) \cos\left(\frac{n\pi\theta}{\alpha}\right) d\theta dr du \right. \\
& \left. - \int_{0}^\pi e^{-\frac{r_0^2\sin^2(u)}{2t}} \int_{0}^\infty re^{-\frac{(r-r_0\cos(u))^2}{2t}} \sum_{n=0}^{\infty} d_n \cos\left(\frac{n\pi\theta_0}{\alpha}\right) \cos\left(\frac{n\pi}{\alpha}u\right) \cdot \int_{0}^\alpha \hat{f}(r,\theta) \cos\left(\frac{n\pi\theta}{\alpha}\right) d\theta dr du \right| \\
& \le \frac{\alpha}{4} \int_0^\infty re^{-\frac{r^2+r_0^2}{2t}} \left( \int_{-\pi}^\pi e^{\frac{rr_0}{t} \cos(u)} |\hat{f}_N(r,\theta_0 + u) - \hat{f}(r,\theta_0 + u)|du \right) dr \\
& \text{where } \hat{f}_N \text{ denotes the } N^{th} \text{ partial Fourier sum of } \hat{f}. \text{ Then, by Cauchy-Schwarz inequality:} \\
& \le \frac{\alpha}{4} \int_0^\infty re^{-\frac{r^2+r_0^2}{2t}} \sqrt{\int_{-\pi}^\pi \left(\hat{f}_N(r,\theta_0+ u) - \hat{f}(r,\theta_0 + u)\right)^2 du} \sqrt{\int_{-\pi}^\pi e^{\frac{2rr_0}{t}\cos(u)}du} dr \\
& \le \frac{\alpha}{4} \sqrt{\left\lceil \frac{\pi}{\alpha} \right\rceil} \sup_r ||\hat{f}_N(r,\cdot) - \hat{f}(r,\cdot)||_2 \underbrace{\int_{0}^\infty re^{-\frac{r^2+r_0^2}{2t}} \sqrt{2\pi} e^{\frac{rr_0}{t}} dr}_{ < + \infty}.
\end{align*}
Using Parseval's equality, we obtain
\begin{align*}
||\hat{f}_N(r,\cdot) - \hat{f}(r,\cdot)||_2^2 & = \sum_{n = N+1}^\infty \int_{-\alpha}^\alpha \cos\left(\frac{n\pi \xi}{\alpha}\right)^2 d\xi \frac{1}{\alpha^2} \left(\int_{-\alpha}^\alpha \hat{f}(r,\theta) \cos\left( \frac{n\pi \theta}{\alpha} \right) d\theta \right)^2 \\
& = \sum_{n = N+1}^\infty \int_{-\alpha}^\alpha \cos\left(\frac{n\pi \xi}{\alpha}\right)^2 d\xi \cdot \frac{1}{n^2\pi^2} \left( \int_{-\alpha}^\alpha \partial_\theta\hat{f}(r,\theta)\sin\left(\frac{n\pi\theta}{\alpha}\right) d\theta \right)^2 \\
& \le \sum_{n = N+1}^\infty  \frac{4\alpha^3}{n^2 \pi^2}  ||\partial_\theta \hat{f}||_\infty.
\end{align*}
Although the extension of $\hat{f}$ on $[-\alpha,\alpha]$ is not an element of $\mathcal{C}^1$, we can perform the integration by parts on each interval $[0,\alpha]$ and $[-\alpha,0]$. Then, $||\hat{f}_N(r,\cdot) - \hat{f}(r,\cdot)||_2\to 0$ as $N \rightarrow \infty$ uniformly in $r$. That way we obtain the convergence to $0$ of the difference between the partial sum and the series.
\end{proof}


\begin{thebibliography}{BSCD13}

\bibitem[Bas96]{bass1996}
Richard~F. Bass.
\newblock Uniqueness for the {S}korokhod equation with normal reflection in
  {L}ipschitz domains.
\newblock {\em Electron. J. Probab.}, 1:no. 11, approx. 29 pp.\, 1996.

\bibitem[BBC05]{bass2005}
Richard~F. Bass, Krzysztof Burdzy, and Zhen-Qing Chen.
\newblock Uniqueness for reflecting {B}rownian motion in lip domains.
\newblock {\em Ann. Inst. H. Poincar\'{e} Probab. Statist.}, 41(2):197--235,
  2005.

\bibitem[BC08]{Burzdy}
Krzysztof Burdzy and Zhen-Qing Chen.
\newblock Discrete approximations to reflected {B}rownian motion.
\newblock {\em Ann. Probab.}, 36(2):698--727, 2008.

\bibitem[BC15]{blanchet2015}
Jose Blanchet and Xinyun Chen.
\newblock Steady-state simulation of reflected {B}rownian motion and related
  stochastic networks.
\newblock {\em Ann. Appl. Probab.}, 25(6):3209--3250, 2015.

\bibitem[BGT04]{Bossy}
Mireille Bossy, Emmanuel Gobet, and Denis Talay.
\newblock A symmetrized {E}uler scheme for an efficient approximation of
  reflected diffusions.
\newblock {\em J. Appl. Probab.}, 41(3):877--889, 2004.

\bibitem[BM18]{blanchet2014}
Jose Blanchet and Karthyek Murthy.
\newblock Exact simulation of multidimensional reflected {B}rownian motion.
\newblock {\em J. Appl. Probab.}, 55(1):137--156, 2018.

\bibitem[BnS97]{Ban}
Rodrigo Ba\~{n}uelos and Robert~G. Smits.
\newblock Brownian motion in cones.
\newblock {\em Probab. Theory Related Fields}, 108(3):299--319, 1997.

\bibitem[BSCD13]{blanchetscalliet2013}
Christophette Blanchet-Scalliet, Areski Cousin, and Diana Dorobantu.
\newblock {Hitting time for correlated three-dimensional Brownian motion}.
\newblock {\em HAL}, 2013.
\newblock hal-00846450v2.

\bibitem[BST10]{Bayer}
Christian Bayer, Anders Szepessy, and Ra\'{u}l Tempone.
\newblock Adaptive weak approximation of reflected and stopped diffusions.
\newblock {\em Monte Carlo Methods Appl.}, 16(1):1--67, 2010.

\bibitem[CBM15]{chupeau2015}
Marie Chupeau, Olivier B\'{e}nichou, and Satya~N. Majumdar.
\newblock Survival probability of a {B}rownian motion in a planar wedge of
  arbitrary angle.
\newblock {\em Phys. Rev. E (3)}, 91(3):032106, 8, 2015.

\bibitem[CJ59]{carslaw1959}
H.~S. Carslaw and J.~C. Jaeger.
\newblock {\em Conduction of Heat in Solids}.
\newblock Oxford University Press, second edition, 1959.

\bibitem[CPS98]{Cos}
C.~Costantini, B.~Pacchiarotti, and F.~Sartoretto.
\newblock Numerical approximation for functionals of reflecting diffusion
  processes.
\newblock {\em SIAM J. Appl. Math.}, 58(1):73--102, 1998.

\bibitem[Die10]{dieker}
A.~B. Dieker.
\newblock Reflected brownian motion.
\newblock {\em Wiley Encyclopedia of Operations Research and Management
  Science}, 2010.

\bibitem[DL06]{lejay2006}
Madalina Deaconu and Antoine Lejay.
\newblock A random walk on rectangles algorithm.
\newblock {\em Methodol. Comput. Appl. Probab.}, 8(1):135--151, 2006.

\bibitem[DM09]{DM}
A.~B. Dieker and J.~Moriarty.
\newblock Reflected {B}rownian motion in a wedge: sum-of-exponential stationary
  densities.
\newblock {\em Electron. Commun. Probab.}, 14:1--16, 2009.

\bibitem[Dub04]{Dubedat}
Julien Dub\'{e}dat.
\newblock Reflected planar {B}rownian motions, intertwining relations and
  crossing probabilities.
\newblock {\em Ann. Inst. H. Poincar\'{e} Probab. Statist.}, 40(5):539--552,
  2004.

\bibitem[EFW13]{escobar2013}
Marcos Escobar, Sebastian Ferrando, and Xianzhang Wen.
\newblock Three dimensional distribution of {B}rownian motion extrema.
\newblock {\em Stochastics}, 85(5):807--832, 2013.

\bibitem[EKM97]{embrechts1997}
Paul Embrechts, Claudia Kl\"{u}ppelberg, and Thomas Mikosch.
\newblock {\em Modelling extremal events}, volume~33 of {\em Applications of
  Mathematics (New York)}.
\newblock Springer-Verlag, Berlin, 1997.
\newblock For insurance and finance.

\bibitem[GN13]{greben}
D.~S. Grebenkov and B.-T. Nguyen.
\newblock Geometrical structure of {L}aplacian eigenfunctions.
\newblock {\em SIAM Rev.}, 55(4):601--667, 2013.

\bibitem[GNR86]{giorno1986}
V.~Giorno, A.~G. Nobile, and L.~M. Ricciardi.
\newblock On some diffusion approximations to queueing systems.
\newblock {\em Adv. in Appl. Probab.}, 18(4):991--1014, 1986.

\bibitem[Gob01]{Gobet}
Emmanuel Gobet.
\newblock Euler schemes and half-space approximation for the simulation of
  diffusion in a domain.
\newblock {\em ESAIM Probab. Statist.}, 5:261--297, 2001.

\bibitem[GR07]{Grad}
I.~S. Gradshteyn and I.~M. Ryzhik.
\newblock {\em Table of integrals, series, and products}.
\newblock Elsevier/Academic Press, Amsterdam, seventh edition, 2007.
\newblock Translated from the Russian, Translation edited and with a preface by
  Alan Jeffrey and Daniel Zwillinger, With one CD-ROM (Windows, Macintosh and
  UNIX).

\bibitem[Gri02]{Grieser}
D.~Grieser.
\newblock Uniform bounds for eigenfunctions of the {L}aplacian on manifolds
  with boundary.
\newblock {\em Comm. Partial Differential Equations}, 27(7-8):1283--1299, 2002.

\bibitem[Ha09]{ha2009}
Wonho Ha.
\newblock {\em Applications of the reflected {O}rnstein-{U}hlenbeck process}.
\newblock ProQuest LLC, Ann Arbor, MI, 2009.
\newblock Thesis (Ph.D.)--University of Pittsburgh.

\bibitem[IKP13]{Ichiba}
Tomoyuki Ichiba, Ioannis Karatzas, and Vilmos Prokaj.
\newblock Diffusions with rank-based characteristics and values in the
  nonnegative quadrant.
\newblock {\em Bernoulli}, 19(5B):2455--2493, 2013.

\bibitem[Iye85]{iyengar1985}
Satish Iyengar.
\newblock Hitting lines with two-dimensional {B}rownian motion.
\newblock {\em SIAM J. Appl. Math.}, 45(6):983--989, 1985.

\bibitem[JYC09]{jeanblanc2010}
Monique Jeanblanc, Marc Yor, and Marc Chesney.
\newblock {\em Mathematical methods for financial markets}.
\newblock Springer Finance. Springer-Verlag London, Ltd., London, 2009.

\bibitem[Kag07]{Kager}
Wouter Kager.
\newblock Reflected {B}rownian motion in generic triangles and wedges.
\newblock {\em Stochastic Process. Appl.}, 117(5):539--549, 2007.

\bibitem[KLR18]{kaushansky2017}
Vadim Kaushansky, Alexander Lipton, and Christoph Reisinger.
\newblock Transition probability of {B}rownian motion in the octant and its
  application to default modelling.
\newblock {\em Appl. Math. Finance}, 25(5-6):434--465, 2018.

\bibitem[Kro92]{kroger}
P.~Kroger.
\newblock Upper bounds for the {N}eumann eigenvalues on a bounded domain in
  {E}uclidean space.
\newblock {\em J. Funct. Anal.}, 106(2):353--357, 1992.

\bibitem[KZ16]{kou_zhong2016}
Steven Kou and Haowen Zhong.
\newblock First-passage times of two-dimensional {B}rownian motion.
\newblock {\em Adv. in Appl. Probab.}, 48(4):1045--1060, 2016.

\bibitem[LSU68]{Lady}
O.~A. Lady\v{z}enskaja, V.~A. Solonnikov, and N.~N. Ural'ceva.
\newblock {\em Linear and quasilinear equations of parabolic type}.
\newblock Translations of Mathematical Monographs, Vol. 23. American
  Mathematical Society, Providence, R.I., 1968.
\newblock Translated from the Russian by S. Smith.

\bibitem[Met09]{metzler2008}
Adam Metzler.
\newblock {\em Multivariate first-passage models in credit risk}.
\newblock ProQuest LLC, Ann Arbor, MI, 2009.
\newblock Thesis (Ph.D.)--University of Waterloo (Canada).

\bibitem[Met10]{metzler2010}
Adam Metzler.
\newblock On the first passage problem for correlated {B}rownian motion.
\newblock {\em Statist. Probab. Lett.}, 80(5-6):277--284, 2010.

\bibitem[MNP00]{Mazja}
Vladimir Maz'ya, Serguei Nazarov, and Boris Plamenevskij.
\newblock {\em Asymptotic theory of elliptic boundary value problems in
  singularly perturbed domains. {V}ol. {I}}, volume 111 of {\em Operator
  Theory: Advances and Applications}.
\newblock Birkh\"{a}user Verlag, Basel, 2000.
\newblock Translated from the German by Georg Heinig and Christian Posthoff.

\bibitem[Pil14]{pilipenko2014}
Andrey Pilipenko.
\newblock {\em An introduction to stochastic differential equations with
  reflection}.
\newblock Number~1 in Lectures in pure and applied mathematics (1).
  Universit\"atsverlag Potsdam, 2014.

\bibitem[RY99]{RY}
Daniel Revuz and Marc Yor.
\newblock {\em Continuous martingales and {B}rownian motion}, volume 293 of
  {\em Grundlehren der mathematischen Wissenschaften [Fundamental Principles of
  Mathematical Sciences]}.
\newblock Springer-Verlag, Berlin, third edition, 1999.

\bibitem[Sai87]{Saisho}
Yasumasa Saisho.
\newblock Stochastic differential equations for multidimensional domain with
  reflecting boundary.
\newblock {\em Probab. Theory Related Fields}, 74(3):455--477, 1987.

\bibitem[TT90]{talay1990}
Denis Talay and Luciano Tubaro.
\newblock Expansion of the global error for numerical schemes solving
  stochastic differential equations.
\newblock {\em Stochastic Anal. Appl.}, 8(4):483--509 (1991), 1990.

\bibitem[VW85]{varadhan1985}
S.~R.~S. Varadhan and R.~J. Williams.
\newblock Brownian motion in a wedge with oblique reflection.
\newblock {\em Comm. Pure Appl. Math.}, 38(4):405--443, 1985.

\bibitem[Wat44]{watson_1944}
G.~N. Watson.
\newblock {\em A {T}reatise on the {T}heory of {B}essel {F}unctions}.
\newblock Cambridge University Press, Cambridge, England; Macmillan Company,
  New York, 1944.

\end{thebibliography}
\end{document}